\documentclass[a4paper,oneside,reqno]{amsart}
\pdfoutput=1
\usepackage[utf8]{inputenc}
\usepackage[
    colorlinks,
    cite color=blue,
    pdfusetitle
    ]{hyperref}
\usepackage{gensymb}
\usepackage{amssymb}
\usepackage[T1]{fontenc}
\usepackage{mathpazo}
\usepackage{microtype}
\usepackage{ellipsis}
\usepackage[english]{babel}
\usepackage{tikz}
\usetikzlibrary{cd}
\usetikzlibrary{backgrounds}
\usepackage{csquotes}
\usepackage[
    backend=bibtex,
    maxalphanames=4,
    minalphanames=3,
    maxbibnames=99,
    style=alphabetic,
    ]{biblatex}
\usepackage{mathtools}
\mathtoolsset{mathic=true}
\usepackage{enumitem}
\usepackage[subrefformat=parens]{subcaption}
\usepackage{amsthm}
\usepackage{thmtools}
\usepackage{cleveref}

\tikzset{posetelm/.style={draw, fill, circle, minimum size=4pt, inner sep=0}}
\tikzset{posetelmm/.style={draw, thick, minimum size=5pt, inner sep=0}}
\tikzset{marking/.style={red}}
\tikzset{elmname/.style={blue}}
\tikzset{covrel/.style={thick}}

\declaretheorem[numberwithin=section]{theorem}
\declaretheorem[numbered=no]{theorem*}
\declaretheorem[sibling=theorem]{proposition}
\declaretheorem[sibling=theorem]{lemma}

\declaretheorem[sibling=theorem]{corollary}
\declaretheorem[style=definition,sibling=theorem]{definition}
\declaretheorem[style=definition,sibling=theorem]{remark}

\declaretheorem[style=definition,sibling=theorem,qed=$\lozenge$]{example}

\newcommand\R{\ensuremath{\mathbb R}}
\newcommand\C{\ensuremath{\mathbb C}}

\newcommand\Z{\ensuremath{\mathbb Z}}
\newcommand\OO{\ensuremath{\mathcal O}}
\newcommand\CC{\ensuremath{\mathcal C}}
\newcommand\FF{\ensuremath{\mathcal F}}

\DeclareMathOperator\rec{rec}
\DeclareMathOperator\im{im}
\DeclareMathOperator\aff{aff}

\DeclareMathOperator\relint{relint}

\newcommand\wt\widetilde
\newcommand\isoto{\xrightarrow{\raisebox{-2pt}[0pt][0pt]{$\mathmakebox[7pt]{\sim}$}}}

\begin{document}

\title[The Face Structure and Geometry of Marked Order Polyhedra]%
      {The Face Structure and Geometry of\\Marked Order Polyhedra}

\author{Christoph Pegel}
\address{Institut für Algebra, Zahlentheorie und Diskrete Mathematik \\
              Leibniz Universität Hannover\\
              Welfengarten 1\\
          30167 Hannover, Germany}
\email{pegel@math.uni-bremen.de}

\begin{abstract}
    We study a class of polyhedra associated to marked posets.
    Examples of these polyhedra are Gelfand--Tsetlin polytopes and cones, as well as Berenstein--Zelevinsky polytopes---all of which have appeared in the representation theory of semi-simple Lie algebras.
    The faces of these polyhedra correspond to certain partitions of the underlying poset and we give a combinatorial characterization of these partitions.
    We specify a class of marked posets that give rise to polyhedra with facets in correspondence to the covering relations of the poset.
    On the convex geometrical side, we describe the recession cone of the polyhedra, discuss products and give a Minkowski sum decomposition.
    We briefly discuss intersections with affine subspaces that have also appeared in representation theory and recently in the theory of finite Hilbert space frames.
\end{abstract}
\maketitle

\section{Introduction}

The mathematical developments leading to marked order polyhedra are split into two separate branches that just recently merged.

The first branch, started by Geissinger and Stanley in the 1980s, comes from order theory and combinatorial convex geometry.
Given a finite poset \(P\) with a global maximum and a global minimum, Geissinger studied the polytope \(\mathcal O(P)\) in \(\R^P\) consisting of order-preserving maps \(P\to\R\) sending the minimum to \(0\) and the maximum to \(1\) in \cite{Geissinger81}.
He found that vertices of this polytope correspond to non-trivial order ideals of \(P\) and more generally, that faces correspond to residually acyclic partitions of \(P\).
Geissinger also describes how the volume of \(\OO(P)\) is obtained from the number of linear extensions of \(P\).
These results reappear in \cite{Stanley86}, where Stanley called \(\OO(P)\) the \emph{order polytope} associated to \(P\) and introduced a second polytope, the \emph{chain polytope} \(\CC(P)\), with inequalities given by saturated chains in \(P\).
He introduced a piecewise linear \emph{transfer map} \(\OO(P)\to\CC(P)\) that yields an Ehrhart equivalence of the polytopes.
In the same spirit of comparing these two polytopes associated to a finite poset, a group around Hibi and Li obtained results on unimodular equivalence and a bijection between edge sets in \cite{HL12} and \cite{HLSS15}, respectively.

A second branch begins in the 1950s in representation theory of semi-simple complex Lie algebras, when Gelfand and Tsetlin introduced number patterns---soon attributed to them as \emph{Gelfand--Tsetlin patterns}---to enumerate the elements in certain bases of irreducible representations in \cite{GT50}.
Given a fixed dominant integral weight \(\lambda\) for the general linear Lie algebra \(\mathfrak{gl}_n(\C)\), the corresponding irreducible representation \(V(\lambda)\) has a basis enumerated by integral Gelfand--Tsetlin patterns.
The defining conditions of these patterns give rise to the \emph{Gelfand--Tsetlin polytope} associated to the dominant weight, so that the elements in the Gelfand--Tsetlin basis correspond to the lattice points in the Gelfand--Tsetlin polytope.
In \cite{GKT13} and \cite{ACK16} the authors used methods from enumerative combinatorics to study the number of vertices and the \(f\)-vector of Gelfand--Tsetlin polytopes, respectively.

Given another weight \(\mu\) of the representation \(V(\lambda)\), one can add certain linear conditions to the description of the Gelfand--Tsetlin polytope to obtain a different polytope, whose integral points enumerate a basis of the weight \(\mu\) subspace of the irreducible representation \(V(\lambda)\).
These polytopes have been studied by De~Loera and McAllister in \cite{DM04}, where they give a procedure to calculate the dimension of minimal faces corresponding to points in the polytope by using a method based on tiling matrices, very similar to the approach we take in this article.
Similar techniques have been used in \cite{Alexandersson16} to study integral points of Gelfand--Tsetlin polytopes.

These Gelfand--Tsetlin polytopes with additional linear conditions recently appeared the theory of finite Hilbert space frames in \cite{CFM11} as \emph{polytopes of eigensteps}.
In this setting they have been used to find parametrizations of frame varieties.
In a special case the authors of \cite{HP16} gave a non-redundant description of the polytope in terms of linear equations and inequalities, hence determining the dimension and number of facets of the polytope.

The branches started to merge in 2011, when Ardila, Bliem and Salazar generalized order and chain polytopes to \emph{marked order and chain polytopes} in \cite{ABS11}, allowing marking conditions other than just sending minima to \(0\) and maxima to \(1\).
That is, given a finite poset \(P\) and a \emph{marking} \(\lambda\colon P^*\to \R\) of a subset \(P^*\subseteq P\) containing all extremal elements, they defined the marked order polytope \(\OO(P,\lambda)\) to be the set of all order-preserving extensions of \(\lambda\) to all of \(P\) and generalized the marked chain polytope \(\CC(P,\lambda)\) accordingly.
This generalization allowed them to consider Gelfand--Tsetlin, Feigin--Fourier--Littelmann--Vinberg and Berenstein--Zelevinsky polytopes---all of which have appeared in representation theory---as marked poset polytopes.
They showed that the transfer map introduced by Stanley generalizes and still yields an Ehrhart equivalence \(\OO(P,\lambda)\to\CC(P,\lambda)\) in the marked case.
In \cite{JS14} the authors study the number of lattice points in \(\OO(P,\lambda)\) when varying the values of \(\lambda\).
They find that this number is piecewise polynomial in the values of \(\lambda\) and use the generalized transfer map to obtain the same result for marked chain polytopes.
They also characterize the faces of marked order polytopes by certain partitions of the underlying poset in the spirit of the original work by Geissinger and Stanley.
However, the characterizations given in \cite[Propositions~2.2,~2.3]{JS14} are incorrect, since the mentioned conditions on the partitions are too weak.
We state the correct characterization in \Cref{thm:facepartitions}.
In \cite{Fourier16} an attempt has been made to define a class of \emph{regular} marked posets, where the facets of the associated marked order polytope are in correspondence with the covering relations of the posets, as is true in the unmarked case.
However, the procedure in \cite[Section~3]{Fourier16} does not remove all redundant covering relations.
We give a corrected definition of regular marked posets in \Cref{def:regular}.
Marked order and chain polytopes have been generalized in \cite{FF16} to an Ehrhart equivalent family of \emph{marked chain-order polytopes}, having marked order and marked chain polytopes as extremal cases.
\medskip

In this article, we restrict our study to a potentially unbounded generalization of marked order polytopes.
We start by defining marked posets and their associated marked order polyhedra in \Cref{sec:mop}, describing different ways to look at the concept from an order theoretic, a convex geometric and a categorical point of view.
We then study the face structure of marked order polyhedra in \Cref{sec:mop-faces} and give a complete combinatorial characterization of partitions of the underlying poset corresponding to faces of the polyhedron.
We specialize this characterization to facets and show that regular marked posets have facets in correspondence with all covering relations of the poset.
In \Cref{sec:mop-geometry}, we focus on convex geometrical properties of marked order polyhedra.
We describe the recession cone of the polyhedra, how disjoint unions of posets correspond to products of polyhedra and give a Minkowski sum decomposition.
Furthermore, we show that marked order polyhedra with integral markings are always lattice polyhedra.
We close by adding linear conditions to marked order polyhedra in \Cref{sec:cmop}, generalizing a result on dimensions of faces obtained in \cite{DM04} for weighted Gelfand--Tsetlin polytopes to these conditional marked order polyhedra.
Throughout the article we give examples to illuminate the obtained results.

\section{Marked Posets and their Associated Polyhedra}
\label{sec:mop}

\begin{definition}
    A \emph{marked poset} \((P,\lambda)\) is a finite poset \(P\) together with an induced subposet \(P^*\subseteq P\) of \emph{marked elements} and an order-preserving \emph{marking} \(\lambda\colon P^*\to\R\).
    The marking \(\lambda\) is called \emph{strict} if \(\lambda(a)<\lambda(b)\) whenever \(a<b\).
    A map \(f\colon (P,\lambda)\to (P',\lambda')\) between marked posets is an order-preserving map \(f\colon P\to P'\) such that \(f(P^*)\subseteq (P')^*\) and \(\lambda'(f(a)) = \lambda(a)\) for all \(a\in P^*\).
\end{definition}

When talking about a poset \(P\) we will always denote its partial order by \(\le\) and covering relations by \(\prec\).
That is, for \(p,q\in P\) we write \(p \prec q\) to indicate that \(p<q\) and \(p\le s\le q\) implies \(s=p\) or \(s=q\).

To study marked posets and the polyhedra we will associate to them, it is sometimes useful to take a more categorically minded point of view on marked posets.
From the definition above we see that marked posets form a category \(\mathsf{MPos}\).
Letting \(\mathsf{Pos}\) denote the category of posets and order-preserving maps, we can describe \(\mathsf{MPos}\) as a category of certain diagrams in \(\mathsf{Pos}\).
A marked poset \((P,\lambda)\) is a diagram
\begin{equation*}
    \begin{tikzcd}
        P & P^* \arrow[r,"\lambda"] \arrow[l, hook'] & \R
    \end{tikzcd}
\end{equation*}
in \(\mathsf{Pos}\), where \(P^*\hookrightarrow P\) is the inclusion of an induced subposet \(P^*\) in a finite poset \(P\).
A map \(f\colon (P,\lambda)\to (P',\lambda')\) is a commutative diagram
\begin{equation*}
    \begin{tikzcd}
        P \arrow[d,"f"] & P^* \arrow[r,"\lambda"] \arrow[l, hook'] \arrow[d, "f|_{P^*}"] & \R \arrow[d,equal]\\
        P' & P'^* \arrow[r,"\lambda'"] \arrow[l, hook'] & \R\makebox[0pt][l]{.}
    \end{tikzcd}
\end{equation*}

To each marked poset \((P,\lambda)\) we associate a polyhedron \(\OO(P,\lambda)\) in \(\R^P\).

\begin{definition} \label{def:mop}
    Let \((P,\lambda)\) be a marked poset.
    The \emph{marked order polyhedron} \(\OO(P,\lambda)\) associated to \((P,\lambda)\) is the set of all \(x\in\R^P\) such that \(x_p\le x_q\) for all \(p,q\in P\) with \(p\le q\) and \(x_a=\lambda(a)\) for all \(a\in P^*\).
\end{definition}

Since the coordinates in \(P^*\) are fixed, denote by $\tilde P=P\setminus P^*$ the remaining coordinates and let \(\wt\OO(P,\lambda)\) be the affinely isomorphic projection of \(\OO(P,\lambda)\) to \(\R^{\tilde P}\).

When \(P^*\) contains all extremal elements of \(P\), the polyhedron \(\OO(P,\lambda)\) is bounded.
In this case \(\wt\OO(P,\lambda)\) is the marked order polytope associated to \((P,\lambda)\) as in \cite{ABS11}.

In more geometric terms, this definition is equivalent to
\begin{equation*}
    \OO(P,\lambda) = \bigcap_{p<q} H_{p<q}^+ \cap \bigcap_{a\in P^*} H_a,
\end{equation*}
where \(H_{p<q}^+\) is the half-space in \(\R^P\) defined by \(x_p\le x_q\) and \(H_a\) is the hyperplane defined by \(x_a=\lambda(a)\).

An interval \([a,b]\) in a marked poset \((P,\lambda)\) is called \emph{constant} if \(a,b\in P^*\) and \(\lambda(a)=\lambda(b)\).
In this case \(x_p=\lambda(a)\) for all \(x\in\OO(P,\lambda)\) and \(p\in[a,b]\).
With this terminology, a marking \(\lambda\) is strict if and only if \((P,\lambda)\) contains no non-trivial constant intervals.

We can also think of the marked order polyhedron \(\OO(P,\lambda)\) as the set of all extensions of \(\lambda\) to order-preserving maps \(x\colon P\to\R\) with \(x|_{P^*} = \lambda\).
That is, the set of all poset maps \(x\colon P\to\R\) such that the diagram
\begin{equation*}
    \begin{tikzcd}
        P \arrow[rr, bend right, "x"] & P^* \arrow[r,"\lambda"] \arrow[l, hook'] & \R
    \end{tikzcd}
\end{equation*}
commutes.
Putting together the diagram of a map \(f\colon (P,\lambda)\to (P',\lambda')\) between marked posets and that of a point \(x\in\OO(P',\lambda')\), we see that we obtain a point \(f^*(x)\) in \(\OO(P,\lambda)\) given by \(f^*(x) = x\circ f\):
\begin{equation*}
     \begin{tikzcd}
        P \arrow[d,"f"] & P^* \arrow[r,"\lambda"] \arrow[l, hook'] \arrow[d, "f|_{P^*}"] & \R \arrow[d,equal]\\
        P' \arrow[rr, bend right, "x"] & P'^* \arrow[r,"\lambda'"] \arrow[l, hook'] & \R\makebox[0pt][l]{.}
    \end{tikzcd}
\end{equation*}
Hence, letting \(\mathsf{Polyh}\) denote the category of polyhedra and affine maps, we have a contravariant functor \(\OO\colon\mathsf{MPos}\to\mathsf{Polyh}\) sending a marked poset \((P,\lambda)\) to the marked order polyhedron \(\OO(P,\lambda)\) and a map \(f\) between marked posets to the induced map \(f^*\) described above.

As we will see in the next proposition, any marking \(\lambda\) can be extended to \(P\) and any strict marking can be extended to a strictly order-preserving map \(P\to\R\).

\begin{proposition} \label{prop:moppoint}
    Let \((P,\lambda)\) be a marked poset.
    The associated marked order polyhedron is non-empty and if \(\lambda\) is strict, there is a point \(x\in \OO(P,\lambda)\) such that \(x_p < x_q\) whenever \(p<q\).
\end{proposition}

\begin{proof}
    The order on \(\R\) is dense and unbounded.
    Hence, whenever \(a<c\) in \(\R\) there is a \(b\in\R\) such that \(a<b<c\) and for any \(b\in\R\) there are \(a,c\in\R\) such that \(a<b<c\).
    Since \(P\) is finite this allows us to successively extend \(\lambda\) to an order-preserving map on \(P\).
    In fact, we can find an order-preserving extension \(x\) of \(\lambda\) such that for \(p<q\) we have \(x_p=x_q\) if and only if there are \(a,b\in P^*\) such that \(a\le p<q\le b\) and \(\lambda(a)=\lambda(b)\).
    In particular, when \(\lambda\) was strict we can always find a strictly order-preserving extension.
\end{proof}

\begin{example} \label{ex:pentagon}
    We consider the marked order polytope given by the marked poset \((P,\lambda)\) in \Cref{subfig:pentagon-a}.
    The blue labels name elements in \(P\), while the red labels correspond to values of the elements of \(P^*\) under the marking \(\lambda\).
    The (projected) associated marked order polytope \(\wt\OO(P,\lambda)\) is shown in \Cref{subfig:pentagon-b}.
    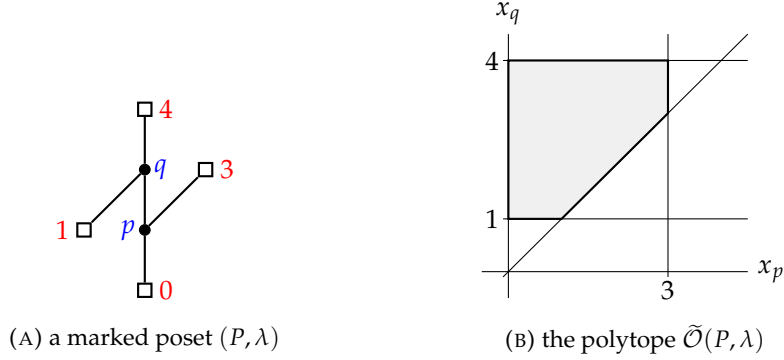
\begin{figure}
        \centering
        \subcaptionbox[]{a marked poset \((P,\lambda)\)\label{subfig:pentagon-a}}[0.49\textwidth][c]{\centering
        \begin{tikzpicture}[baseline={([yshift=-.5ex]current bounding box.center)},scale=0.8]
            \path (0,0) node[posetelmm] (0) {} node[right=2pt,marking] {\(0\)};
            \path (0,1) node[posetelm] (x) {} node[left,elmname] {\(p\)};
            \path (0,2) node[posetelm] (y) {} node[right,elmname] {\(q\)};
            \path (0,3) node[posetelmm] (4) {} node[right=2pt,marking] {\(4\)};
            \path (-1,1) node[posetelmm] (1) {} node[left=2pt,marking] {\(1\)};
            \path (1,2) node[posetelmm] (3) {} node[right=2pt,marking] {\(3\)};
            \draw[covrel]
                  (0) -- (x) -- (y) -- (4)
                  (1) -- (y)
                  (x) -- (3);
        \end{tikzpicture}}
        \hfill
        \subcaptionbox[]{the polytope \(\wt\OO(P,\lambda)\)\label{subfig:pentagon-b}}[0.49\textwidth][c]{\centering
        \begin{tikzpicture}[scale=.7,baseline={([yshift=-.5ex]current bounding box.center)}]
            \draw (-.5,0) -- (4.5,0) node [right] {\(x_p\)};
            \draw (0,-.5) -- (0,4.5) node [above] {\(x_q\)};
            \draw (-.1,4) -- (4.5,4);
            \draw (-.1,1) -- (4.5,1);
            \draw (3,-.1) -- (3,4.5);
            \draw (-.1,-.1) -- (4.5,4.5);
            \fill[black!6,draw=black,thick] (0,1) -- (1,1) -- (3,3) -- (3,4) -- (0,4) -- cycle;
            \draw (3,0) node[below] {\(3\)};
            \draw (0,4) node[left] {\(4\)};
            \draw (0,1) node[left] {\(1\)};
        \end{tikzpicture}}
        \caption[Marked poset and associated polytope from \Cref{ex:pentagon}]{The marked poset \((P,\lambda)\) from \Cref{ex:pentagon} and the associated marked order polytope \(\wt\OO(P,\lambda)\).}
        \label{fig:pentagon}
    \end{figure}
\end{example}

\section{Face Structure and Facets}
\label{sec:mop-faces}

In this section, we study the face structure of \(\OO(P,\lambda)\).
As it turns out, the faces of marked order polyhedra correspond to certain partitions of the underlying poset \(P\).
Our goal is to characterize those partitions combinatorially.
We associate to each point \(x\) in \(\OO(P,\lambda)\) a partition \(\pi_x\) of \(P\), that will suffice to describe the minimal face of \(\OO(P,\lambda)\) containing \(x\).
The partitions that are obtained in this way from points of the polyhedron will then---ordered by refinement---capture the polyhedrons face structure.

\begin{definition}
    Let \(Q=\OO(P,\lambda)\) be a marked order polyhedron.
    To each \(x\in Q\) we associate a partition \(\pi_x\) of \(P\) induced by the transitive closure of the relation
    \begin{equation*}
        p \sim_x q \quad\text{if}\quad \text{\(x_p=x_q\) and \(p,q\) are comparable.}
    \end{equation*}
    We may think of \(\pi\) as being obtained by first partitioning \(P\) into blocks of constant values under \(x\) and then splitting those blocks into connected components with respect to the Hasse diagram of \(P\).

    Given any partition \(\pi\) of \(P\), we call a block \(B\in\pi\) \emph{free} if \(P^*\cap B = \varnothing\) and denote by \(\tilde\pi\) the set of all free blocks of \(\pi\).
    Note that any \(x\in\OO(P,\lambda)\) is constant on the blocks of \(\pi_x\) and the values on the non-free blocks of \(\pi_x\) are determined by \(\lambda\).
\end{definition}

Let \(x\in Q\) be a point of a polyhedron.
We denote the minimal face of \(Q\) containing \(x\) by \(F_x\).
Hence, \(F_x\) is the unique face having \(x\) in its relative interior.
Equivalently, \(F_x\) is the intersection of all faces of \(Q\) containing \(x\).

\begin{proposition} \label{prop:mopface}
    Let \(x\in Q=\OO(P,\lambda)\) be a point of a marked order polyhedron with associated partition \(\pi=\pi_x\).
    We have
    \begin{equation*}
        F_x = \left\{ \, y\in Q : \text{\(y\) is constant on the blocks of \(\pi\)} \,\right\}
    \end{equation*}
    and \(\dim F_x = |\tilde\pi|\).
\end{proposition}

\begin{proof}
    For \(p<q\) in \(P\) let \(H_{p<q}=\partial H_{p<q}^+\) be the hyperplane defined by \(x_p=x_q\) in \(\R^P\).
    The minimal face of a point \(x\in Q\) is then given by
    \begin{equation*}
        F_x = Q \ \cap\  \smashoperator{\bigcap_{\substack{p<q,\\x_p=x_q}}} H_{p<q}.
    \end{equation*}
    A point \(y\in Q\) satisfies \(y_p=y_q\) for all \(p<q\) with \(x_p=x_q\) if and only if \(y\) is constant on the blocks of \(\pi_x\).
    Thus, \(F_x\) is indeed given by all \(y\in Q\) constant on the blocks of \(\pi_x\).

    To determine the dimension of \(F_x\), we consider its affine hull \(\aff(F_x)\).
    It is obtained by intersecting the affine hull of \(Q\) with all \(H_{p<q}\) such that \(x_p=x_q\).
    The affine hull of \(Q\) itself is the intersection of all \(H_a\) for \(a\in P^*\) and all \(H_{p<q}\) such that \(y_p=y_q\) for all \(y\in Q\).
    Putting these facts together, we have
    \begin{equation*}
        \aff(F_x) = \bigcap_{a\in P^*} H_a \enskip\cap\enskip \smashoperator{\bigcap_{\substack{p<q,\\y_p=y_q \forall y\in Q}}} H_{p<q} \enskip\cap\enskip \smashoperator{\bigcap_{\substack{p<q,\\x_p=x_q}}} H_{p<q}
        = \bigcap_{a\in P^*} H_a \enskip\cap\enskip \smashoperator{\bigcap_{\substack{p<q,\\x_p=x_q}}} H_{p<q}.
    \end{equation*}
    This is exactly the set of all \(y\) constant on the blocks of \(\pi_x\) and satisfying \(y_a=\lambda(a)\) for all \(a\in P^*\).
    Such \(y\) are uniquely determined by values on the free blocks of \(\pi_x\) and thus \(\dim(F_x)=|\tilde\pi_x|\) as desired.
\end{proof}

\begin{corollary}
    If \(\lambda\) is a strict marking on \(P\), the dimension of \(\OO(P,\lambda)\) is equal to the number of unmarked elements in \(P\).
\end{corollary}

\begin{proof}
    Since all coordinates in \(P^*\) are fixed by \(\lambda\), we always have \(\dim\OO(P,\lambda)\le |\tilde P|\).
    If \(\lambda\) is strict, there is a point \(x\in\OO(P,\lambda)\) such that \(x_p<x_q\) whenever \(p<q\) by \Cref{prop:moppoint}.
    Hence, \(\pi_x\) is the partition of \(P\) into singletons and \(\dim F_x = |\tilde\pi_x| = |\tilde P|\).
    We conclude that \(F_x=\OO(P,\lambda)\), so \(x\) is a relative interior point and the marked order polyhedron has the desired dimension.
\end{proof}

\begin{corollary}
    Let \(x\in Q=\OO(P,\lambda)\) be a point of a marked order polyhedron.
    For \(y\in Q\) we have \(y\in F_x\) if and only if \(\pi_x\) is a refinement of \(\pi_y\).
\end{corollary}

\begin{proof}
    By \Cref{prop:mopface}, \(y\in F_x\) if and only if \(y\) is constant on the blocks of \(\pi_x\).
    Let \(y\) be constant on the blocks of \(\pi_x\).
    Any block \(B\) of \(\pi_x\) is connected with respect to the Hasse diagram of \(P\) and \(y\) takes constant values on \(B\), hence \(B\) is contained in a block of \(\pi_y\) by construction and \(\pi_x\) is a refinement of \(\pi_y\).
    Now let \(y\in Q\) with \(\pi_x\) being a refinement of \(\pi_y\).
    We conclude that \(y\) is constant on the blocks of \(\pi_x\), since it is constant on the blocks of \(\pi_y\) and \(\pi_x\) is a refinement of \(\pi_y\).
\end{proof}

\begin{corollary}
    Given any two points \(x,y\in\OO(P,\lambda)\), we have \(F_y\subseteq F_x\) if and only if \(\pi_x\) is a refinement of \(\pi_y\).
    In particular \(F_y=F_x\) if and only if \(\pi_y=\pi_x\). \qed
\end{corollary}

Hence, the partition of \(\OO(P,\lambda)\) into relative interiors of its faces is the same as the partition given by \(x\sim y\) if \(\pi_x=\pi_y\) and we can associate to each non-empty face \(F\) a partition \(\pi_F\) with \(\pi_F=\pi_x\) for any \(x\) in the relative interior of \(F\).
We call a partition \(\pi\) of \(P\) a \emph{face partition} of \((P,\lambda)\) if \(\pi=\pi_F\) for some non-empty face of \(\OO(P,\lambda)\).
We arrive at the following description of face lattices of marked order polyhedra.

\begin{corollary}
    Let \(Q=\OO(P,\lambda)\) be a marked order polyhedron.
    The poset \(\FF(Q)\setminus\{\varnothing\}\) of non-empty faces of \(Q\) is isomorphic to the induced subposet of the partition lattice on \(P\) given by all face partitions of \((P,\lambda)\). \qed
\end{corollary}

For the marked order polytope from \Cref{ex:pentagon}, we illustrated the face partitions in \Cref{fig:facepartitions}.
The free blocks are highlighted in blue round shapes, non-free blocks in red angular shapes.
We see that the dimensions of the faces are given by the numbers of free blocks in the associated face partitions and that face inclusions correspond to refinements of partitions.

\begin{figure}
    \centering
\begin{tikzpicture}[scale=1.5]
    \tikzset{el/.style={draw, fill, circle, minimum size=2pt, inner sep=0, outer sep=0}}
    \tikzset{elm/.style={fill, minimum size=2pt, inner sep=0, outer sep=0}}
    \tikzset{blockfree/.style={blue,line width=5.6pt, line join=round, line cap=round}}
    \tikzset{blocknonfree/.style={red,line width=4.6pt, line cap=rect}}
    \tikzset{block2free/.style={blue!20,fill,line width=4.3pt, line join=round, line cap=round}}
    \tikzset{block2nonfree/.style={red!20,line width=3.3pt, line cap=rect}}
    \draw (0,1) coordinate (A);
    \draw (1,1) coordinate (B);
    \draw (3,3) coordinate (C);
    \draw (3,4) coordinate (D);
    \draw (0,4) coordinate (E);

    \fill[fill=black!6,draw=black,thick] (A) -- (B) -- (C) -- (D) -- (E) -- cycle;

    \newcommand\drawblock[2][]{
                \draw[fill,block#1] #2;
                \draw[fill,block2#1] #2;
    }

    \draw (A)++(-.4,-.4) node[anchor=center] {
        \begin{tikzpicture}[scale=.3]
            \path (0,0) node[elm] (0) {};
            \path (0,1) node[el] (x) {};
            \path (0,2) node[el] (y) {};
            \path (0,3) node[elm] (4) {};
            \path (-1,1) node[elm] (1) {};
            \path (1,2) node[elm] (3) {};
            \draw[covrel]
                  (0.center) -- (x.center) -- (y.center) -- (4.center)
                  (1.center) -- (y.center)
                  (x.center) -- (3.center);
            \begin{scope}[on background layer]
                \drawblock[nonfree]{(0.center) -- (x.center)}
                \drawblock[nonfree]{(1.center) -- (y.center)}
                \drawblock[nonfree]{(3.center) -- +(0,.1pt)}
                \drawblock[nonfree]{(4.center) -- +(0,.1pt)}
            \end{scope}
        \end{tikzpicture}
    };
    \draw (B)++(.4,-.4) node[anchor=center] {
        \begin{tikzpicture}[scale=.3]
            \path (0,0) node[elm] (0) {};
            \path (0,1) node[el] (x) {};
            \path (0,2) node[el] (y) {};
            \path (0,3) node[elm] (4) {};
            \path (-1,1) node[elm] (1) {};
            \path (1,2) node[elm] (3) {};
            \draw[covrel]
                  (0.center) -- (x.center) -- (y.center) -- (4.center)
                  (1.center) -- (y.center)
                  (x.center) -- (3.center);
            \begin{scope}[on background layer]
                \drawblock[nonfree]{(0.center) -- +(0,.1pt)}
                \drawblock[nonfree]{(1.center) -- (y.center) -- (x.center) -- cycle}
                \drawblock[nonfree]{(3.center) -- +(0,.1pt)}
                \drawblock[nonfree]{(4.center) -- +(0,.1pt)}
            \end{scope}
        \end{tikzpicture}
    };
    \draw (C)++(.4,-.4) node[anchor=center] {
        \begin{tikzpicture}[scale=.3]
            \path (0,0) node[elm] (0) {};
            \path (0,1) node[el] (x) {};
            \path (0,2) node[el] (y) {};
            \path (0,3) node[elm] (4) {};
            \path (-1,1) node[elm] (1) {};
            \path (1,2) node[elm] (3) {};
            \draw[covrel]
                  (0.center) -- (x.center) -- (y.center) -- (4.center)
                  (1.center) -- (y.center)
                  (x.center) -- (3.center);
            \begin{scope}[on background layer]
                \drawblock[nonfree]{(0.center) -- +(0,.1pt)}
                \drawblock[nonfree]{(1.center) -- +(0,.1pt)}
                \drawblock[nonfree]{(x.center) -- (y.center) -- (3.center) -- cycle}
                \drawblock[nonfree]{(4.center) -- +(0,.1pt)}
            \end{scope}
        \end{tikzpicture}
    };
    \draw (D)++(.4,.4) node[anchor=center] {
        \begin{tikzpicture}[scale=.3]
            \path (0,0) node[elm] (0) {};
            \path (0,1) node[el] (x) {};
            \path (0,2) node[el] (y) {};
            \path (0,3) node[elm] (4) {};
            \path (-1,1) node[elm] (1) {};
            \path (1,2) node[elm] (3) {};
            \draw[covrel]
                  (0.center) -- (x.center) -- (y.center) -- (4.center)
                  (1.center) -- (y.center)
                  (x.center) -- (3.center);
            \begin{scope}[on background layer]
                \drawblock[nonfree]{(0.center) -- +(0,.1pt)}
                \drawblock[nonfree]{(1.center) -- +(0,.1pt)}
                \drawblock[nonfree]{(x.center) -- (3.center)}
                \drawblock[nonfree]{(y.center) -- (4.center)}
            \end{scope}
        \end{tikzpicture}
    };
    \draw (E)++(-.4,.4) node[anchor=center] {
        \begin{tikzpicture}[scale=.3]
            \path (0,0) node[elm] (0) {};
            \path (0,1) node[el] (x) {};
            \path (0,2) node[el] (y) {};
            \path (0,3) node[elm] (4) {};
            \path (-1,1) node[elm] (1) {};
            \path (1,2) node[elm] (3) {};
            \draw[covrel]
                  (0.center) -- (x.center) -- (y.center) -- (4.center)
                  (1.center) -- (y.center)
                  (x.center) -- (3.center);
            \begin{scope}[on background layer]
                \drawblock[nonfree]{(0.center) -- (x.center)}
                \drawblock[nonfree]{(1.center) -- +(0,.1pt)}
                \drawblock[nonfree]{(y.center) -- (4.center)}
                \drawblock[nonfree]{(3.center) -- +(0,.1pt)}
            \end{scope}
        \end{tikzpicture}
    };
    \draw (0,2.5)++(-.4,0) node[anchor=center] {
        \begin{tikzpicture}[scale=.3]
            \path (0,0) node[elm] (0) {};
            \path (0,1) node[el] (x) {};
            \path (0,2) node[el] (y) {};
            \path (0,3) node[elm] (4) {};
            \path (-1,1) node[elm] (1) {};
            \path (1,2) node[elm] (3) {};
            \draw[covrel]
                  (0.center) -- (x.center) -- (y.center) -- (4.center)
                  (1.center) -- (y.center)
                  (x.center) -- (3.center);
            \begin{scope}[on background layer]
                \drawblock[nonfree]{(0.center) -- (x.center)}
                \drawblock[free]{(y.center) -- (y.center)}
                \drawblock[nonfree]{(1.center) -- +(0,.1pt)}
                \drawblock[nonfree]{(3.center) -- +(0,.1pt)}
                \drawblock[nonfree]{(4.center) -- +(0,.1pt)}
            \end{scope}
        \end{tikzpicture}
    };
    \draw (0.5,1)++(0,-.4) node[anchor=center] {
        \begin{tikzpicture}[scale=.3]
            \path (0,0) node[elm] (0) {};
            \path (0,1) node[el] (x) {};
            \path (0,2) node[el] (y) {};
            \path (0,3) node[elm] (4) {};
            \path (-1,1) node[elm] (1) {};
            \path (1,2) node[elm] (3) {};
            \draw[covrel]
                  (0.center) -- (x.center) -- (y.center) -- (4.center)
                  (1.center) -- (y.center)
                  (x.center) -- (3.center);
            \begin{scope}[on background layer]
                \drawblock[nonfree]{(0.center) -- +(0,.1pt)}
                \drawblock[free]{(x.center) -- (x.center)}
                \drawblock[nonfree]{(1.center) -- (y.center)}
                \drawblock[nonfree]{(3.center) -- +(0,.1pt)}
                \drawblock[nonfree]{(4.center) -- +(0,.1pt)}
            \end{scope}
        \end{tikzpicture}
    };
    \draw (2,2)++(.4,-.4) node[anchor=center] {
        \begin{tikzpicture}[scale=.3]
            \path (0,0) node[elm] (0) {};
            \path (0,1) node[el] (x) {};
            \path (0,2) node[el] (y) {};
            \path (0,3) node[elm] (4) {};
            \path (-1,1) node[elm] (1) {};
            \path (1,2) node[elm] (3) {};
            \draw[covrel]
                  (0.center) -- (x.center) -- (y.center) -- (4.center)
                  (1.center) -- (y.center)
                  (x.center) -- (3.center);
            \begin{scope}[on background layer]
                \drawblock[nonfree]{(0.center) -- +(0,.1pt)}
                \drawblock[free]{(x.center) -- (y.center)}
                \drawblock[nonfree]{(1.center) -- +(0,.1pt)}
                \drawblock[nonfree]{(3.center) -- +(0,.1pt)}
                \drawblock[nonfree]{(4.center) -- +(0,.1pt)}
            \end{scope}
        \end{tikzpicture}
    };
    \draw (3,3.5)++(.4,0) node[anchor=center] {
        \begin{tikzpicture}[scale=.3]
            \path (0,0) node[elm] (0) {};
            \path (0,1) node[el] (x) {};
            \path (0,2) node[el] (y) {};
            \path (0,3) node[elm] (4) {};
            \path (-1,1) node[elm] (1) {};
            \path (1,2) node[elm] (3) {};
            \draw[covrel]
                  (0.center) -- (x.center) -- (y.center) -- (4.center)
                  (1.center) -- (y.center)
                  (x.center) -- (3.center);
            \begin{scope}[on background layer]
                \drawblock[nonfree]{(0.center) -- +(0,.1pt)}
                \drawblock[free]{(y.center) -- (y.center)}
                \drawblock[nonfree]{(1.center) -- +(0,.1pt)}
                \drawblock[nonfree]{(x.center) -- (3.center)}
                \drawblock[nonfree]{(4.center) -- +(0,.1pt)}
            \end{scope}
        \end{tikzpicture}
    };
    \draw (1.5,4)++(0,.4) node[anchor=center] {
        \begin{tikzpicture}[scale=.3]
            \path (0,0) node[elm] (0) {};
            \path (0,1) node[el] (x) {};
            \path (0,2) node[el] (y) {};
            \path (0,3) node[elm] (4) {};
            \path (-1,1) node[elm] (1) {};
            \path (1,2) node[elm] (3) {};
            \draw[covrel]
                  (0.center) -- (x.center) -- (y.center) -- (4.center)
                  (1.center) -- (y.center)
                  (x.center) -- (3.center);
            \begin{scope}[on background layer]
                \drawblock[nonfree]{(0.center) -- +(0,.1pt)}
                \drawblock[free]{(x.center) -- (x.center)}
                \drawblock[nonfree]{(1.center) -- +(0,.1pt)}
                \drawblock[nonfree]{(3.center) -- +(0,.1pt)}
                \drawblock[nonfree]{(y.center) -- (4.center)}
            \end{scope}
        \end{tikzpicture}
    };
    \draw (1.3,2.7) node[anchor=center] {
        \begin{tikzpicture}[scale=.3]
            \path (0,0) node[elm] (0) {};
            \path (0,1) node[el] (x) {};
            \path (0,2) node[el] (y) {};
            \path (0,3) node[elm] (4) {};
            \path (-1,1) node[elm] (1) {};
            \path (1,2) node[elm] (3) {};
            \draw[covrel]
                  (0.center) -- (x.center) -- (y.center) -- (4.center)
                  (1.center) -- (y.center)
                  (x.center) -- (3.center);
            \begin{scope}[on background layer]
                \drawblock[nonfree]{(0.center) -- +(0,.1pt)}
                \drawblock[free]{(x.center) -- (x.center)}
                \drawblock[free]{(y.center) -- (y.center)}
                \drawblock[nonfree]{(1.center) -- +(0,.1pt)}
                \drawblock[nonfree]{(3.center) -- +(0,.1pt)}
                \drawblock[nonfree]{(4.center) -- +(0,.1pt)}
            \end{scope}
        \end{tikzpicture}
    };
\end{tikzpicture}
\caption[Face partitions of \Cref{ex:pentagon}]{The face partitions of the marked order polytope in \Cref{ex:pentagon}.}
    \label{fig:facepartitions}
\end{figure}
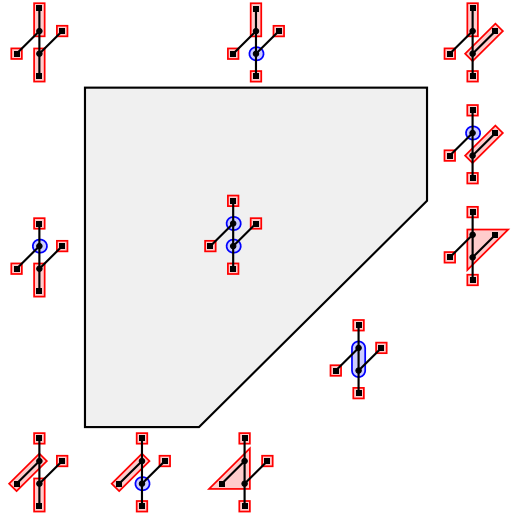

In order to characterize the face partitions of a marked poset \((P,\lambda)\) combinatorially, we introduce some properties of partitions of \(P\).

\begin{definition}\label{def:mpquotient}
    Let \((P,\lambda)\) be a marked poset.
    A partition \(\pi\) of \(P\) is \emph{connected} if the blocks of \(\pi\) are connected as induced subposets of \(P\).
    It is \emph{\(P\)-compatible}, if the relation \(\le\) defined on \(\pi\) as the transitive closure of
    \begin{equation*}
        B \le C \quad\text{if}\quad \text{\(p\le q\) for some \(p\in B\), \(q\in C\)}
    \end{equation*}
    is anti-symmetric.
    In this case \(\le\) is a partial order on \(\pi\).
    A \(P\)-compatible partition \(\pi\) is called \emph{\((P,\lambda)\)-compatible}, if whenever \(a\in B\cap P^*\) and \(b\in C\cap P^*\) for some blocks \(B\le C\), we have \(\lambda(a)\le \lambda(b)\).
\end{definition}

\begin{remark}
    Whenever a partition \(\pi\) of a poset \(P\) is \(P\)-compatible, it is also \emph{convex}.
    That is, for \(a<b<c\) with \(a\) and \(c\) in the same block \(B\in\pi\), we also have \(b\in B\), since otherwise the blocks containing \(a\) and \(b\) would contradict the relation on the blocks being anti-symmetric.
    This implies that the blocks in a connected, \(P\)-compatible partition are not just connected as induced subposets of \(P\) but even connected as induced subgraphs of the Hasse diagram of \(P\).
\end{remark}

\begin{proposition} \label{prop:mp-quotient}
    Let \((P,\lambda)\) be a marked poset.
    A \((P,\lambda)\)-compatible partition \(\pi\) of \(P\) gives rise to a marked poset \((P/\pi, \lambda/\pi)\) where \(P/\pi\) is the poset of blocks in \(\pi\), \((P/\pi)^* = \pi\setminus\tilde\pi\) and \(\lambda/\pi\colon (P/\pi)^*\to\R\) is defined by \((\lambda/\pi)(B) = \lambda(a)\) for any \(a\in B\cap P^*\).
    Furthermore, the quotient map \(P\to P/\pi\) defines a map \((P,\lambda)\to(P/\pi,\lambda/\pi)\) of marked posets.
\end{proposition}

\begin{proof}
    Since \(\pi\) is \(P\)-compatible, the blocks of \(\pi\) form a poset \(P/\pi\) as in \Cref{def:mpquotient}.
    Since \(\pi\) is \((P,\lambda)\)-compatible, we have \(\lambda(a)=\lambda(b)\) whenever \(a,b\in B\cap P^*\) for some non-free block \(B\in\pi\).
    Hence, the map \(\lambda/\pi\) is well-defined.
    It is order-preserving by the definition of \((P,\lambda)\)-compatibility.
    Furthermore, we have a commutative diagram
\begin{equation*}
    \begin{tikzcd}
        P \arrow[d] & P^* \arrow[r,"\lambda"] \arrow[l, hook'] \arrow[d] & \R \arrow[d,equal]\\
        P/\pi & (P/\pi)^* \arrow[r,"\lambda/\pi"] \arrow[l, hook'] & \R\makebox[0pt][l]{.}
    \end{tikzcd}
\end{equation*}
Thus, we have a quotient map \((P,\lambda)\to(P/\pi,\lambda/\pi)\).
\end{proof}

\begin{proposition} \label{prop:facepartitionproperties}
    Every face partition \(\pi_F\) of \((P,\lambda)\) is \((P,\lambda)\)-compatible, connected and the induced marking on \((P/\pi_F,\lambda/\pi_F)\) is strict.
\end{proposition}

\begin{proof}
    Let \(F\) be a non-empty face of \(\OO(P,\lambda)\).
    It is obvious that \(\pi_F\) is connected by construction, since it is given by the transitive closure of a relation that only relates pairs of comparable elements.
    To verify that \(\pi_F\) is \(P\)-compatible, we need to check that the induced relation \(\le\) on the blocks of \(\pi_F\) is anti-symmetric.
    Assume we have blocks \(B,C\in\pi_F\) such that \(B\le C\) and \(C\le B\).
    Since \(B\le C\), there is a finite sequence of blocks \(B=X_1,X_2,\dots,X_k,X_{k+1}=C\) such that for \(i=1,\dots,k\) there are some \(p_i\in X_i\), \(q_i\in X_{i+1}\) with \(p_i\le q_i\).
    Take any \(x\) in the relative interior of \(F\), then \(x_{p_i} \le x_{q_i}\) for \(i=1,\dots,k\) and since \(x\) is constant on the blocks of \(\pi_F\), we have \(x_{q_i}=x_{p_{i+1}}\) for \(i=1,\dots,k-1\).
    To summarize, we have
    \begin{equation}
        x_{p_1} \le x_{q_1} = x_{p_2} \le x_{q_2} = \cdots \le \cdots = x_{p_k} \le x_{q_k}.
        \label{eq:blockchainvalues}
    \end{equation}
    Hence, the constant value \(x\) takes on \(B\) is less than or equal to the constant value \(x\) takes on \(C\).
    Since we also have \(C\le B\), we conclude that \(x\) takes equal values on the blocks \(B\) and \(C\).
    From \eqref{eq:blockchainvalues} we conclude that \(x\) takes equal values on all blocks \(X_i\).
    From the definition of \(\pi_x=\pi_F\) it follows that the blocks \(X_i\) are in fact all equal, in particular \(B=C\) and the relation is anti-symmetric.

    To see that \(\pi_F\) is \((P,\lambda)\)-compatible, let \(B,C\in\pi\) be non-free blocks with \(B\le C\).
    By the same argument as above, we know that any \(x\in F\) has constant value on \(B\) less than or equal to the constant value on \(C\), so \(\lambda(a)\le\lambda(b)\) for marked \(a\in B\), \(b\in C\).
    If \(\lambda(a)=\lambda(b)\) we have \(B=C\), by the same argument as above, so the induced marking is strict.
\end{proof}

Given any partition \(\pi\) of \(P\), we can define a polyhedron \(F_\pi\) contained in \(\OO(P,\lambda)\) by
\begin{equation*}
    F_\pi = \left\{ \, y\in Q : \text{\(y\) is constant on the blocks of \(\pi\)} \right\}.
\end{equation*}
If \(\pi=\pi_F\) is a face partition of \((P,\lambda)\), we have \(F_\pi = F\) by \Cref{prop:mopface}.
However, \(F_\pi\) is not a face for all partitions \(\pi\) of \(P\).

As long as \(\pi\) is \((P,\lambda)\)-compatible, we can show that the polyhedron \(F_\pi\) is affinely isomorphic to the marked order polyhedron \(\OO(P/\pi, \lambda/\pi)\).
The isomorphism will be induced by the quotient map \(P\to P/\pi\).
Our first step is to verify that this induced map is indeed an injection.

\begin{lemma} \label{lem:f*inj}
    Let \(f\colon (P,\lambda)\to(P',\lambda')\) be a map of marked posets.
    If \(f\) is surjective, the induced map \(f^*\colon \OO(P',\lambda')\to \OO(P,\lambda)\) is injective.
\end{lemma}

\begin{proof}
    Let \(x,y\in \OO(P',\lambda')\) such that \(f^*(x)=f^*(y)\).
    Given any \(p\in P'\) we need to show \(x_p=y_p\).
    Since \(f\) is surjective, \(p=f(q)\) for some \(q\in P\) and thus
    \begin{equation*}
        x_p = x_{f(q)} = f^*(x)_q = f^*(y)_q = y_{f(q)} = y_p. \qedhere
    \end{equation*}
\end{proof}

\begin{proposition} \label{prop:submop}
    Let \((P,\lambda)\) be a marked poset and \(\pi\) a \((P,\lambda)\)-compatible partition.
    The quotient map \(q\colon (P,\lambda)\to(P/\pi,\lambda/\pi)\) induces an injection
    \begin{equation*}
        q^*\colon \OO(P/\pi,\lambda/\pi) \lhook\joinrel\longrightarrow \OO(P,\lambda)
    \end{equation*}
    with image \(q^*(\OO(P/\pi,\lambda/\pi)) = F_\pi\).
\end{proposition}

\begin{proof}
    By \Cref{lem:f*inj} we know that \(q^*\) is an injection.
    Hence, we only need to verify that \(F_\pi\) is the image of \(q^*\).
    The image is contained in \(F_\pi\), since whenever \(p\) and \(p'\) are in the same block \(B\in\pi\), we have
    \begin{equation*}
        q^*(x)_p = x_{q(p)} = x_B = x_{q(p')} = q^*(x)_{p'}.
    \end{equation*}
    Hence, all \(q^*(x)\) are constant on the blocks of \(\pi\).
    Conversely, given any point \(y\in\OO(P,\lambda)\) constant on the blocks of \(\pi\), we obtain a well defined map \(x\colon P/\pi \to \R\) sending each block to the constant value \(y_p\) for all \(p\) in the block.
    This map is a point \(x\in\OO(P/\pi,\lambda/\pi)\) mapped to \(y\) by \(q^*\).
\end{proof}

The previous proposition tells us, that whenever we have a \((P,\lambda)\)-compatible partition \(\pi\), the marked order polyhedron \(\OO(P/\pi,\lambda/\pi)\) is affinely isomorphic to the polyhedron \(F_\pi\subseteq\OO(P,\lambda)\) via the embedding \(q^*\) induced by the quotient map.
From now on, we refer to affine isomorphisms arising this way as the \emph{canonical affine isomorphism \(\OO(P/\pi,\lambda/\pi)\cong F_\pi\)}.

\begin{corollary}
    For every non-empty face \(F\) of a marked order polyhedron \(\OO(P,\lambda)\) we have a canonical affine isomorphism \(\OO(P/\pi_F,\lambda/\pi_F)\cong F\). \qed
\end{corollary}

We are now ready to state and prove the characterization of face partitions of marked posets.

\begin{theorem} \label{thm:facepartitions}
    A partition \(\pi\) of a marked poset \((P,\lambda)\) is a face partition if and only if it is \((P,\lambda)\)-compatible, connected and the induced marking on \((P/\pi,\lambda/\pi)\) is strict.
\end{theorem}

\begin{proof}
    The fact that face partitions satisfy the above properties is the statement of \Cref{prop:facepartitionproperties}.
    Now let \(\pi\) be a partition of \(P\) that is \((P,\lambda)\)-compatible, connected and induces a strict marking \(\lambda/\pi\).
    By \Cref{prop:moppoint}, there is a point \(z\in\OO(P/\pi,\lambda/\pi)\) such that \(z_B < z_C\) whenever \(B<C\).
    Let \(x\in\R^P\) be the point in the polyhedron \(F_\pi\subseteq\OO(P,\lambda)\) obtained as the image of \(z\) under the canonical affine isomorphism \(\OO(P/\pi,\lambda/\pi)\isoto F_\pi\).
    We claim that \(\pi=\pi_x\), so \(\pi\) is a face partition.
    Since \(x\) is constant on the blocks of \(\pi\) and \(\pi\) is connected, we know that \(\pi\) is a refinement of \(\pi_x\).
    Now assume that the equivalence relation \(\sim_x\) defining \(\pi_x\) relates elements in different blocks of \(\pi\).
    In this case, there are blocks \(B\neq C\) of \(\pi\) with elements \(p\in B\), \(q\in C\) such that \(x_p=x_q\) and \(p<q\).
    This implies that \(z_B=z_C\) and \(B<C\), a contradiction to the choice of \(z\).
    Hence, \(\pi=\pi_x\) and \(\pi\) is a face partition of \((P,\lambda)\).
\end{proof}

\begin{remark} \label{rem:combtype}
    To decide whether a given partition \(\pi\) of a marked poset \((P,\lambda)\) satisfies the conditions in \Cref{thm:facepartitions}, it is enough to know the linear order on \(\lambda(P^*)\).
    The exact values of the marking are irrelevant.
    Hence, the face lattice of \(\OO(P,\lambda)\) is determined solely by discrete, combinatorial data.
    In fact, since the directions of facet normals do not depend on the values of \(\lambda\), we can conclude that the normal fan \(\mathcal N(\OO(P,\lambda))\) is determined by this combinatorial data.
    However, the affine isomorphism type of \(\OO(P,\lambda)\) does depend on the exact values of \(\lambda\).
\end{remark}

\begin{remark}
    When $P$ is a poset with a global minimum $\hat 0$ and global maximum $\hat 1$, the marked order polytope $\OO(P,\lambda)$ for $\lambda\colon\hat 0\mapsto 0, \hat 1\mapsto 1$ is the usual order polytope $\OO(P)$ as discussed by Stanley and Geissinger.
    In this case $(P,\lambda)$-compatibility of a partition $\pi$ of $P$ is equivalent to $\pi$ being $P$-compatible and $\hat 0$ and $\hat 1$ being in different blocks.
    Hence, the only $P$-compatible partition that is not $(P,\lambda)$-compatible is the trivial partition with a single block.
    Furthermore, the induced marking on $(P/\pi,\lambda/\pi)$ is always strict.
    Thus, we recover the face description of $\OO(P)$ in terms of connected, compatible partitions given in \cite[Theorem~1.2]{Stanley86}.
\end{remark}

\begin{example} \label{ex:afftype}
    We construct a continuous family \((Q_t)_{t\in[0,1]}\) of marked order polytopes, whose underlying marked posets all yield the same combinatorial data in the sense of \Cref{rem:combtype}, but \(Q_s\) and \(Q_t\) are affinely isomorphic if and only if \(s=t\).
    Let \((P,\lambda_t)\) be the marked poset shown in \Cref{subfig:Qt-a}.
    Letting \(t\) vary in \([0,1]\), we obtain for each \(t\) a different affine isomorphism type, since two of the vertices of \(Q_t\) will move, while the other three stay fixed and are affinely independent as can be seen in \Cref{subfig:Qt-b}.
    However, all \(Q_t\) share the same normal fan and are in particular combinatorially equivalent.
    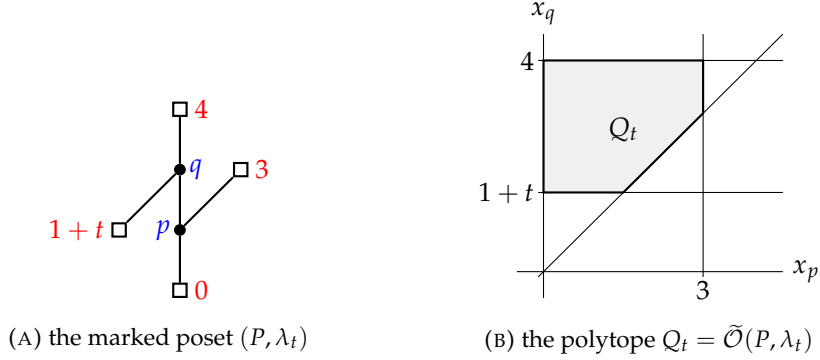
\begin{figure}
        \centering
        \subcaptionbox[]{the marked poset \((P,\lambda_t)\)\label{subfig:Qt-a}}[0.49\textwidth][c]{\centering
        \begin{tikzpicture}[baseline={([yshift=-.5ex]current bounding box.center)},scale=0.8]
            \path (0,0) node[posetelmm] (0) {} node[right=2pt,marking] {\(0\)};
            \path (0,1) node[posetelm] (x) {} node[left,elmname] {\(p\)};
            \path (0,2) node[posetelm] (y) {} node[right,elmname] {\(q\)};
            \path (0,3) node[posetelmm] (4) {} node[right=2pt,marking] {\(4\)};
            \path (-1,1) node[posetelmm] (1) {} node[left=2pt,marking] {\(1+t\)};
            \path (1,2) node[posetelmm] (3) {} node[right=2pt,marking] {\(3\)};
            \draw[covrel]
                  (0) -- (x) -- (y) -- (4)
                  (1) -- (y)
                  (x) -- (3);
        \end{tikzpicture}}
        \hfill
        \subcaptionbox[]{the polytope \(Q_t=\wt\OO(P,\lambda_t)\)\label{subfig:Qt-b}}[0.49\textwidth][c]{\centering
        \begin{tikzpicture}[baseline={([yshift=-.5ex]current bounding box.center)},scale=.7]
            \draw (-.5,0) -- (4.5,0) node [right] {\(x_p\)};
            \draw (0,-.5) -- (0,4.5) node [above] {\(x_q\)};
            \draw (-.1,4) -- (4.5,4);
            \draw (-.1,1.5) -- (4.5,1.5);
            \draw (3,-.1) -- (3,4.5);
            \draw (-.1,-.1) -- (4.5,4.5);
            \fill[black!6,draw=black,thick] (0,1.5) -- (1.5,1.5) -- (3,3) -- (3,4) -- (0,4) -- cycle;
            \draw node at (1.5, 2.7) {\(Q_t\)};
            \draw (3,0) node[below] {\(3\)};
            \draw (0,4) node[left] {\(4\)};
            \draw (0,1.5) node[left] {\(1+t\)};
        \end{tikzpicture}}
        \caption[Marked poset and associated polytope from \Cref{ex:afftype}]{The marked poset \((P,\lambda_t)\) from \Cref{ex:afftype} and the associated marked order polytope \(Q_t=\wt\OO(P,\lambda_t)\).}
        \label{fig:Qt}
    \end{figure}
\end{example}

We continue our study of the face structure of marked order polyhedra by having a closer look at facets.
Since inequalities in the description of marked order polyhedra come from covering relations in the underlying poset, we expect a correspondence of facets to certain covering relations.
If the marked poset satisfies a certain regularity condition, the facets are indeed in bijection with the covering relations.
Hence, if we can change the underlying poset of a marked order polyhedron to a regular one, without changing the associated polyhedron, we obtain an enumeration of facets.
We start by modifying an arbitrary marked poset to a strict one by contracting constant intervals.

\begin{proposition} \label{prop:strictify}
    Given any marked poset \((P,\lambda)\), the partition \(\pi\) induced by the relations \(a\sim p\) and \(p\sim b\) whenever \([a,b]\) is a constant interval containing \(p\) yields a strictly marked poset \((P/\pi,\lambda/\pi)\) such that \(\OO(P/\pi,\lambda/\pi)\cong F_\pi=\OO(P,\lambda)\) via the canonical affine isomorphism.
\end{proposition}

\begin{proof}
    Let \(x\in\OO(P,\lambda)\) be a point constructed as in the proof of \Cref{prop:moppoint}.
    By construction we have \(x_p=x_q\) for \(p<q\) if and only if there are \(a,b\in P^*\) with \(a\le p<q\le b\) with \(\lambda(a)=\lambda(b)\).
    Thus, we conclude that \(\pi_x=\pi\) and \(\pi\) is a face partition of \(\OO(P,\lambda)\).
    Since every point of \(\OO(P,\lambda)\) satisfies \(x_a=x_p=x_b\) whenever \([a,b]\) is a constant interval containing \(p\), we conclude that \(F_\pi\) is indeed the whole polyhedron.
    Hence, \(\OO(P/\pi,\lambda/\pi)\cong F_\pi = \OO(P,\lambda)\), where \(\lambda/\pi\) is a strict marking by \Cref{prop:facepartitionproperties}.
\end{proof}

\begin{definition} \label{def:regular}
    Let \((P,\lambda)\) be a marked poset.
    A covering relation \(p\prec q\) is called \emph{non-redundant} if for all marked elements \(a,b\) satisfying \(a\le q\) and \(p\le b\), we have \(a=b\) or \(\lambda(a) < \lambda(b)\).
    Otherwise the covering relation is called \emph{redundant}.
    The marked poset \((P,\lambda)\) is called \emph{regular}, if all its covering relations are non-redundant.
\end{definition}

Apart from the desired correspondence of covering relations and facets, regularity of marked posets implies some useful properties of the marked poset itself.

\begin{proposition} \label{prop:regular}
    Let \((P,\lambda)\) be a regular marked poset.
    The following conditions are satisfied:
    \begin{enumerate}[nosep,label=\roman*)]
        \item the marking \(\lambda\) is strict,
        \item there are no covering relations between marked elements,
        \item every element in \(P\) covers and is covered by at most one marked element.
    \end{enumerate}
\end{proposition}

\begin{proof}
    \begin{enumerate}[nosep,label=\roman*)]
        \item When \(a<b\) are marked elements of \(P\), there is some covering relation \(p\prec q\) such that \(a\le p\prec q\le b\).
            Since \(a\le q\) and \(p\le b\), we have \(\lambda(a) < \lambda(b)\) by regularity.
        \item When \(b\prec a\) is a covering relation between marked elements, we have \(\lambda(a)<\lambda(b)\) by choosing \(p=b\), \(q=a\) in the regularity condition.
            This is a contradiction to \(\lambda\) being order-preserving.
        \item When \(a,b\prec q\) for marked \(a,b\), the regularity condition for \(a\le q\) and \(b\le b\) implies \(a=b\) or \(\lambda(a)<\lambda(b)\).
            By the same argument we get \(a=b\) or \(\lambda(b)<\lambda(a)\).
            We conclude that \(a=b\). \qedhere
    \end{enumerate}
\end{proof}

\begin{remark} \label{rem:counterex}
    The conditions in \Cref{prop:regular} are necessary, but not sufficient for \((P,\lambda)\) to be regular.
    The marked poset in \Cref{subfig:counterex-a} satisfies all three conditions, but the covering relation \(p\prec q\) is redundant.
    \begin{figure}
        \centering
        \subcaptionbox[]{a marked poset \((P,\lambda)\)\label{subfig:counterex-a}}[0.49\textwidth][c]{\centering
        \begin{tikzpicture}[baseline={([yshift=-.5ex]current bounding box.center)},scale=0.8]
            \path (0,0) node[posetelmm] (0) {} node[right=2pt,marking] {\(0\)};
            \path (0,1) node[posetelm] (x) {} node[left,elmname] {\(p\)};
            \path (0,2) node[posetelm] (y) {} node[right,elmname] {\(q\)};
            \path (0,3) node[posetelmm] (3) {} node[right=2pt,marking] {\(3\)};
            \path (-1,1) node[posetelmm] (2) {} node[left=2pt,marking] {\(2\)};
            \path (1,2) node[posetelmm] (1) {} node[right=2pt,marking] {\(1\)};
            \draw[covrel]
                  (0) -- (x) -- (y) -- (3)
                  (2) -- (y)
                  (x) -- (1);
        \end{tikzpicture}}
        \hfill
        \subcaptionbox[]{the polytope \(\wt\OO(P,\lambda)\)\label{subfig:counterex-b}}[0.49\textwidth][c]{\centering
        \begin{tikzpicture}[scale=.8,baseline={([yshift=-.5ex]current bounding box.center)}]
            \draw (-.5,0) -- (2.5,0) node[right] {\(x_p\)};
            \draw (0,-.5) -- (0,3.5) node[above] {\(x_q\)};
            \fill[black!6,draw=black,thick] (0,2) rectangle (1,3);
            \draw (0,-.1) -- (0,3.5);
            \draw (1,-.1) -- (1,3.5);
            \draw (-.1,-.1) -- (2.5,2.5);
            \draw (-.1,2) -- (2.5,2);
            \draw (-.1,3) -- (2.5,3);
            \draw[latex-] (1.5,1.5) to[bend right] (2.5,1.2) node [right] {\(x_p = x_q\)};
            \draw (1,0) node[below] {\(1\)};
            \draw (0,2) node[left] {\(2\)};
            \draw (0,3) node[left] {\(3\)};
        \end{tikzpicture}}
        \caption[Marked poset and associated polytope from \Cref{rem:counterex}]{The marked poset \((P,\lambda)\) from \Cref{rem:counterex} and the associated marked order polytope \(\wt\OO(P,\lambda)\). The covering relation \(p\prec q\) is redundant.}
        \label{fig:counterex}
    \end{figure}
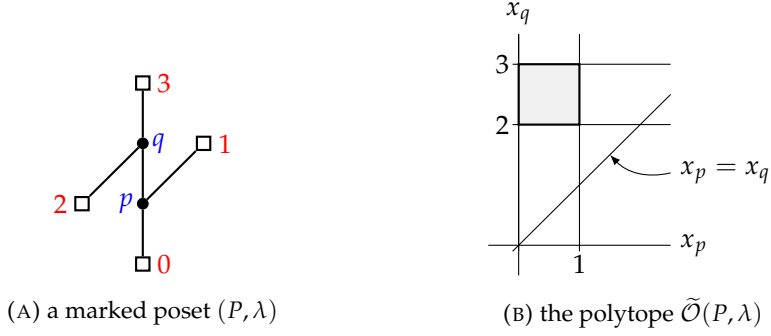

    In fact, this example shows that the process described by Fourier in \cite[Section~3]{Fourier16} does not remove all redundant covering relations and hence leads to a notion of regularity that is not sufficient to have facets in correspondence with covering relations.

    The same marked poset also serves as a counterexample to the characterization of face partitions in \cite[Proposition~2.3]{JS14}.
    Instead of partitions of \(P\) in terms of blocks, they use subposets of \(P\) that have all the elements of \(P\) but only some of the relations.
    The connected components of the Hasse diagram of such a subposet \(G\) give a connected partition \(\pi_G\) of \(P\) and conversely every connected partition \(\pi\) defines a subposet \(G_\pi\) on the elements of \(P\) by having \(p\le q\) in \(G_\pi\) if and only if \(p\le q\) in \(P\) and \(p\) and \(q\) are in the same block of \(\pi\).
    When \((P,\lambda)\) is the marked poset in \Cref{fig:counterex} and \(G\) is the subposet with \(p\le q\) as the only non-reflexive relation---i.e., \(\{p,q\}\) is the only non-singleton block in \(\pi_G\)---the conditions in Proposition~2.3 of \cite{JS14} are satisfied but \(G\) does not yield a face of \(\OO(P,\lambda)\) as can be seen in \Cref{subfig:counterex-b}.
\end{remark}

\begin{theorem}
    Let \((P,\lambda)\) be a regular marked poset.
    The facets of \(\OO(P,\lambda)\) correspond to the covering relations in \((P,\lambda)\).
\end{theorem}

\begin{proof}
    Since \((P,\lambda)\) is strictly marked, the dimension of \(\OO(P,\lambda)\) is equal to the number of unmarked elements in \(P\).
    Hence, a facet \(F\) corresponds to a \((P,\lambda)\)-compatible, connected partition \(\pi\) of \(P\) such that \(\lambda/\pi\) is strict and \(\pi\) has exactly \(|\tilde P|-1\) free blocks.
    We claim that the number of non-free blocks of \(\pi\) is \(|P^*|\).
    Assume there are marked elements \(a\neq b\) in a common block \(B\) of \(\pi\).
    Since \(\pi\) has \(|\tilde P|-1\) free blocks, at most one unmarked element can be in a non-free block.
    Since \((P,\lambda)\) is regular, there are no covering relations between marked elements.
    Hence, since \(B\) is connected as an induced subgraph of the Hasse diagram of \(P\) and contains both \(a\) and \(b\), it also contains the only unmarked element \(p\) in a non-free block, and we have one of the following four situations: \(a\prec p \prec b\),\, \(a \succ p \succ b\),\, \(a \prec p \succ b\) or \(a \succ p \prec b\).
    Since \(a\) and \(b\) are in the same block, they are identically marked and the first two possibilities contradict \(\lambda\) being strict.
    The other two possibilities contradict regularity, since \(p\) covers---or is covered by---more than one marked element.
    Hence, \(\pi\) has exactly \(|P^*|\) non-free blocks and we conclude that \(\pi\) has \(|P|-1\) blocks overall.
    Therefore, \(\pi\) consists of \(|P|-2\) singletons and a single connected 2-element block corresponding to a covering relation of \(P\).

    Conversely, let \(p\prec q\) be a covering relation of \(P\).
    We claim that the partition \(\pi\) with the only non-singleton block \(\{p,q\}\) is a face partition with \(|\tilde P|-1\) free blocks.
    Since \((P,\lambda)\) is regular, it contains no covering relation between marked elements and \(\pi\) has exactly \(|\tilde P|-1\) free blocks.
    Since \(\{p,q\}\) is the only non-singleton block and \(p\prec q\), the partition \(\pi\) is connected and \(P\)-compatible.
    To verify that \(\pi\) is \((P,\lambda)\)-compatible and \(\lambda/\pi\) is strict, let \(B, C\) be non-free blocks of \(\pi\) with \(a\in B\cap P^*\) and \(b\in C\cap P^*\) such that \(B\le C\).
    When \(B=C\), we have \(a=b\) and \(\lambda(a)=\lambda(b)\).
    When \(B<C\), we conclude \(a<b\) or \(a\le q\), \(p\le b\), since \(\{p,q\}\) is the only non-trivial block.
    In both cases, regularity implies \(\lambda(a)<\lambda(b)\).
\end{proof}

\begin{remark}
    In the case of usual order polytopes $\OO(P)$ every covering relation is non-redundant, since the only markings are $\lambda(\hat 0)=0$ and $\lambda(\hat 1)=1$.
    Hence, the notion of regularity is superfluous in the absence of a marking.
\end{remark}

Now that we established a regularity condition on marked posets that guarantees a bijection between covering relations in \(P\) and facets of the marked order polyhedron, we explain how to transform any given marked poset to a regular one.

\begin{proposition} \label{prop:regularize}
    Let \((P,\lambda)\) be a strictly marked poset.
    Redundant covering relations in \(P\) can be removed successively to obtain a regular marked poset \((P',\lambda)\) with the same associated marked order polyhedron \(\OO(P',\lambda)=\OO(P,\lambda)\).
\end{proposition}

\begin{proof}
    Let \(p\prec q\) be a redundant covering relation in \(P\).
    That is, there are marked elements \(a\neq b\) satisfying \(a\le q\), \(p\le b\) and \(\lambda(a)\ge \lambda(b)\).
    Let \(P'\) be obtained from \(P\) by removing the covering relation \(p\prec q\) from \(P\).
    Obviously \(\OO(P,\lambda)\) is contained in \(\OO(P',\lambda)\).
    
    Now let \(x\in\OO(P',\lambda)\).
    To verify that \(x\) is a point of \(\OO(P,\lambda)\), we have to show \(x_p\le x_q\).
    Since \(\lambda\) is a strict marking on \(P\), we can not have \(a\le p\).
    Otherwise \(a\le p\le b\) implies \(a<b\), in contradiction to \(\lambda(a)\ge \lambda(b)\).
    Hence, removing the covering relation \(p\prec q\) we still have \(a\le' q\) in \(P'\).
    By the same argument \(p \le' b\).
    Thus, by the defining conditions of \(\OO(P',\lambda)\), we have
    \begin{equation*}
        x_p \le x_b = \lambda(b) \le \lambda(a) = x_a \le x_q.
    \end{equation*}
    Therefore, \(x\in\OO(P,\lambda)\) and we conclude \(\OO(P',\lambda)=\OO(P,\lambda)\).
    This process can be repeated until all redundant covering relations have been removed, resulting in a regular marked poset defining the same marked order polyhedron.
\end{proof}

\begin{remark}
    Note that \Cref{prop:regularize} does not imply, that all covering relations that are redundant in \((P,\lambda)\) can be removed simultaneously.
    Removing a single redundant covering relation can lead to other redundant covering relations becoming non-redundant.
    In the marked poset
    
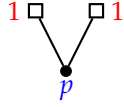
\begin{equation*}
        \begin{tikzpicture}[xscale=.4,yscale=.8,baseline={([yshift=-.5ex]current bounding box.center)}]
            \path (-1,1) node[posetelmm] (1) {} node[left=2pt,marking] {\(1\)};
            \path (0,0) node[posetelm] (p) {} node[below,elmname] {\(p\)};
            \path (1,1) node[posetelmm] (1') {} node[right=2pt,marking] {\(1\)};
            \draw[covrel]
                  (1) -- (p) -- (1');
        \end{tikzpicture}
    \end{equation*}
    both covering relations are redundant.
    However, removing any of the two covering relations renders the remaining covering relation non-redundant.
\end{remark}

Given any marked poset \((P,\lambda)\), we can apply the constructions of \Cref{prop:strictify} and \Cref{prop:regularize} to obtain a regular marked poset \((P',\lambda')\) defining the same marked order polyhedron up to canonical affine isomorphism.

\section{Products, Minkowski Sums and Lattice Polyhedra}
\label{sec:mop-geometry}

In this section we study some convex geometric properties of marked order polyhedra.
We describe recession cones, a correspondence between disjoint unions of posets and products of polyhedra, characterize pointedness and use these results to obtain a Minkowski sum decomposition.
At the end of the section we show that marked posets with integral markings always give rise to lattice polyhedra.

\begin{proposition} \label{prop:rec}
    The recession cone of \(\OO(P,\lambda)\) is \(\OO(P,0)\), where \(0\colon P^*\to \R\) is the zero marking on the same domain as \(\lambda\).
\end{proposition}

\begin{proof}
    The recession cone of a polyhedron \(Q\subseteq\R^n\) defined by a system of linear inequalities \(Ax \ge b\) is given by \(Ax \ge 0\).
    Hence, replacing all constant terms in the description of \(\OO(P,\lambda)\) by zeros, we see that \(\rec(\OO(P,\lambda)) = \OO(P,0)\).
\end{proof}

\begin{proposition} \label{prop:products}
    Let \((P_1,\lambda_1)\) and \((P_2,\lambda_2)\) be marked posets on disjoint sets.
    Let the marking \(\lambda_1\sqcup\lambda_2\colon P_1^*\sqcup P_2^*\to\R\) on \(P_1\sqcup P_2\) be given by \(\lambda_1\) on \(P_1^*\) and \(\lambda_2\) on \(P_2^*\).
    The marked order polyhedron \(\OO(P_1\sqcup P_2, \lambda_1\sqcup\lambda_2)\) is equal to the product \(\OO(P_1,\lambda_1)\times\OO(P_2,\lambda_2)\) under the canonical identification \(\R^{P_1\sqcup P_2} = \R^{P_1}\times\R^{P_2}\).
\end{proposition}

\begin{proof}
    The defining equations and inequalities of a product polyhedron \(Q_1\times Q_2\) in \(\R^{P_1}\times\R^{P_2}\) are obtained by imposing both the defining conditions of \(Q_1\) and \(Q_2\).
    In case of \(Q_1=\OO(P_1,\lambda_1)\) and \(Q_2=\OO(P_2,\lambda_2)\) these are exactly the defining conditions of \(\OO(P_1\sqcup P_2,\lambda_1\sqcup \lambda_2)\).
\end{proof}

Note that this relation between disjoint unions of marked posets and products of the associated marked order polyhedra may be expressed as the contravariant functor \(\OO\colon\mathsf{MPos}\to\mathsf{Polyh}\) sending coproducts to products.
\medskip

We now characterize marked posets whose associated polyhedra are \emph{pointed}
A pointed polyhedron is one that has at least one vertex, or equivalently does not contain a line.
The importance of pointedness lies in the fact that pointed polyhedra are determined by their vertices and recession cone.
To be precise, a pointed polyhedron is the Minkowski sum of its recession cone and the polytope obtained as the convex hull of its vertices.

\begin{proposition} \label{prop:pointed}
    A marked order polyhedron \(\OO(P,\lambda)\) is pointed if and only if each connected component of \(P\) contains a marked element.
\end{proposition}

\begin{proof}
    Let \(P_1,\dots,P_k\) be the connected components of \(P\) with \(\lambda_i=\lambda|_{P_i}\) the restricted markings.
    By inductively applying \Cref{prop:products}, we have a decomposition
    \begin{equation*}
        \OO(P,\lambda) = \OO(P_1,\lambda_1) \times \cdots \times \OO(P_k,\lambda_k).
    \end{equation*}
    Hence, \(\OO(P,\lambda)\) is pointed if and only if each \(\OO(P_i,\lambda_i)\) is pointed, reducing the statement to the case of \(P\) being connected.

    Let \((P,\lambda)\) be a connected marked poset and suppose \(v\in\OO(P,\lambda)\) is a vertex.
    By \Cref{prop:mopface} the corresponding partition \(\pi\) has no free blocks.
    Hence, either \(P\) is empty or it has at least as many marked elements as the number of blocks in \(\pi\).
    
    Conversely, if \(P\) is connected and contains marked elements, the following procedure yields a vertex \(v\) of \(\OO(P,\lambda)\):
    start by setting \(v_a = \lambda(a)\) for all \(a\in P^*\).
    Pick any \(p\in P\) such that \(v_p\) is not already determined and \(p\) is adjacent to some \(q\) in the Hasse-diagram of \(P\) with \(v_q\) already determined.
    Set \(v_p\) to be the maximum of all determined \(v_q\) with \(p\) covering \(q\) or the minimum of all determined \(v_q\) with \(p\) covered by \(q\).
    Continue until all \(v_p\) are determined.
    
    In each step, the defining conditions of \(\OO(P,\lambda)\) are respected and the procedure determines all \(v_p\) since \(P\) is connected and contains a marked element.
    By construction, each block of \(\pi_v\) will contain a marked element and thus \(v\) is a vertex by \Cref{prop:mopface}.
\end{proof}

\begin{proposition} \label{prop:minkowski-a}
Let \(\lambda_1, \lambda_2\colon P^* \to \R\) be markings on the same poset \(P\). The Minkowski sum \(\OO(P,\lambda_1)+\OO(P,\lambda_2)\) is contained in \(\OO(P,\lambda_1+\lambda_2)\), where \(\lambda_1+\lambda_2\) is the marking sending \(a\in P^*\) to \(\lambda_1(a)+\lambda_2(a)\).
\end{proposition}

\begin{proof}
    Let \(x\in\OO(P,\lambda_1)\) and \(y\in\OO(P,\lambda_2)\).
    For any relation \(p\le q\) in \(P\) we have \(x_p\le x_q\) and \(y_p\le y_q\), hence \(x_p+y_p \le x_q+y_q\).
    For \(a\in P^*\) we have \(x_a+y_a=\lambda_1(a)+\lambda_2(a) = (\lambda_1+\lambda_2)(a)\).
    Thus, \(x+y\in\OO(P,\lambda_1+\lambda_2)\).
\end{proof}

We are now ready to give a Minkowski sum decomposition of marked order polyhedra, such that the marked posets associated to the summands have \(0\)-\(1\)-markings.
The decomposition is a generalization of \cite[Theorem~4]{Stanley02} and \cite[Corollary~2.10]{JS14}, where the bounded case with \(P^*\) being a chain in \(P\) is considered.

\begin{theorem} \label{thm:minkowski-b}
    Let \((P,\lambda)\) be a marked poset with \(P^*\neq\varnothing\) and \(\lambda(P^*) = \{ \, c_0, c_1, \dots, c_k \, \}\) with \(c_0 < c_1 < \cdots < c_k\).
    Let \(c_{-1}=0\) and define markings \(\lambda_i\colon P^*\to\R\) for \(i=0,\dots,k\) by
    \begin{equation*}
        \lambda_i(a) =
        \begin{cases}
            0 & \text{if \(\lambda(a)<c_i\)}, \\
            1 & \text{if \(\lambda(a)\ge c_i\).}
        \end{cases}
    \end{equation*}
    Then \(\OO(P,\lambda)\) decomposes as the weighted Minkowski sum
    \begin{equation*}
        \OO(P,\lambda) = \sum_{i=0}^k (c_i-c_{i-1}) \, \OO(P,\lambda_i).
    \end{equation*}
\end{theorem}

\begin{proof}
    Since
    \begin{equation*}
        \lambda = c_0 \lambda_0 + (c_1-c_0) \lambda_1 + \cdots + (c_k - c_{k-1}) \lambda_k
    \end{equation*}
    and in general \(\OO(P,c\lambda)=c\,\OO(P,\lambda)\), one inclusion follows immediately from \Cref{prop:minkowski-a}.
    For the other inclusion, first assume that \(\OO(P,\lambda)\) is pointed.
    In this case, it is enough to consider vertices and the recession cone.
    Since the underlying posets and sets of marked elements agree for all polytopes in consideration, they all have the same recession cone \(\OO(P,0)\) by \Cref{prop:rec}.
    Let \(v\in \OO(P,\lambda)\) be a vertex.
    The associated face partition \(\pi\) has no free blocks and on each block \(v\) takes some constant value in \(\lambda(P^*)\).
    For fixed \(i\in\{0,\dots,k\}\) we enumerate the blocks of \(\pi\) where \(v\) takes constant value \(c_i\) by \(B_{i,1},\dots,B_{i,r_i}\).
    For a block \(B\in\pi\) denote by \(w_B = \sum_{p\in B} e_p \in \R^P\) the labeling of \(P\) with all entries in \(B\) equal to \(1\), all other entries equal to \(0\).
    This yields a description of \(v\) as
    \begin{equation*}
        v = \sum_{i=0}^k c_i \sum_{j=1}^{r_i} w_{B_{i,j}}.
    \end{equation*}
    For \(i=0,\dots,k\) define points \(v^{(i)}\in\R^P\) by
    \begin{equation*}
        v^{(i)} = (c_i - c_{i-1}) \sum_{l=i}^k \sum_{j=1}^{r_l} w_{B_{l,j}}.
    \end{equation*}
    This gives a decomposition of \(v\) as \(v^{(0)}+\cdots+v^{(k)}\).
    It remains to be checked that each \(v^{(i)}\) is a point in the corresponding Minkowski summand.
    Since \(v^{(0)}\) is just constant \(c_0\) on the whole poset and \(\lambda_0\) is the marking of all ones, we have \(v^{(0)}\in c_0\,\OO(P,\lambda_0)\).
    Fix \(i\in\{1,\dots,k\}\).
    For \(p\le q\) we have \(v_p\le v_q\) and thus \(p\in B_{i,j}\), \(q\in B_{i',j'}\) for \(i\le i'\) by the chosen enumeration of blocks.
    Hence, by definition of \(v^{(i)}\), the inequality \(v^{(i)}_p \le v^{(i)}_q\) is equivalent to one of the three inequalities \(0\le 0\), \(0\le c_i-c_{i-1}\) or \(c_i - c_{i-1} \le c_i - c_{i-1}\), all being true.
    The marking conditions of \(\OO(P,(c_i-c_{i-1})\lambda_i)\) are satisfied by \(v^{(i)}\) as well, so \(v^{(i)}\in (c_i-c_{i-1})\,\OO(P,\lambda_i)\).
    We conclude that
    \begin{equation*}
        v = \sum_{i=1}^k v^{(i)} \enskip\in\enskip \sum_{i=0}^k (c_i-c_{i-1})\,\OO(P,\lambda_i)
    \end{equation*}
    for each vertex \(v\) of \(\OO(P,\lambda)\).
    Hence, the proof is finished for the case of \(\OO(P,\lambda)\) being pointed.

    When \(\OO(P,\lambda)\) is not pointed, we can decompose \(P=P'\sqcup P''\) where \(P'\) consists of all connected components without marked elements and \(P''\) consists of all other components.
    Letting \(\lambda'\) and \(\lambda''\) be the respective restrictions of \(\lambda\), we have \(\OO(P,\lambda)=\OO(P',\lambda')\times\OO(P'',\lambda'')\) by \Cref{prop:products}, where \(\OO(P',\lambda')\) is not pointed while \(\OO(P'',\lambda'')\) is, by \Cref{prop:pointed}.
    Applying the previous result to \(\OO(P'',\lambda'')\) we obtain
    \begin{align*}
        \OO(P,\lambda) &= \OO(P',\lambda') \times \left( \sum_{i=0}^k (c_i-c_{i-1})\,\OO(P'',\lambda''_i) \right).
    \end{align*}
    Since \(P'\) contains no marked elements, it is equal to its recession cone and we have
    \begin{equation*}
        \OO(P',\lambda') = \sum_{i=0}^k \OO(P',\lambda').
    \end{equation*}
    Therefore, using the identity \(\sum_{i=0}^k P_i \times \sum_{i=0}^k Q_i  = \sum_{i=0}^k (P_i\times Q_i)\) for products of Minkowski sums, we obtain
    \begin{align*}
        \OO(P,\lambda) &= \left( \sum_{i=0}^k \OO(P',\lambda') \right) \times \left( \sum_{i=0}^k (c_i-c_{i-1})\,\OO(P'',\lambda''_i) \right) \\
        &= \sum_{i=0}^k \left( \OO(P',\lambda') \times \OO(P'',(c_i-c_{i-1}) \lambda''_i) \right) \\
        &= \sum_{i=0}^k \OO(P'\sqcup P'', \lambda'\sqcup (c_i-c_{i-1}) \lambda''_i).
    \end{align*}
    Since \(P'\) did non contain any markings that could be affected by scaling, the factors \((c_i-c_{i-1})\) can be put as dilation factors in front of the polyhedra.
    Again, since \(P'\) is unmarked, we have \(\lambda'\sqcup\lambda''_i = \lambda_i\) and \(P'\sqcup P''=P\), so we obtain the desired Minkowski sum decomposition.
\end{proof}

\begin{remark}
    When \(\OO(P,\lambda)\) is a polytope, \(\OO(P,1)\) is just a point and the marked poset polytopes \(\OO(P,\lambda_i)\) appearing in the Minkowski sum decomposition of \Cref{thm:minkowski-b} may all be expressed as ordinary poset polytopes as discussed by Stanley \cite{Stanley86} and Geissinger \cite{Geissinger81} by contracting constant intervals and dropping redundant conditions.
\end{remark}

\begin{example}
    We apply the Minkowski sum decomposition of \Cref{thm:minkowski-b} to the marked order polytope \((P,\lambda)\) from \Cref{ex:pentagon}.
    Since \(\lambda(P^*) = \{0,1,3,4\}\) in this example, we obtain the four new markings \(\lambda_0\), \(\lambda_1\), \(\lambda_2\) and \(\lambda_3\) given by
    
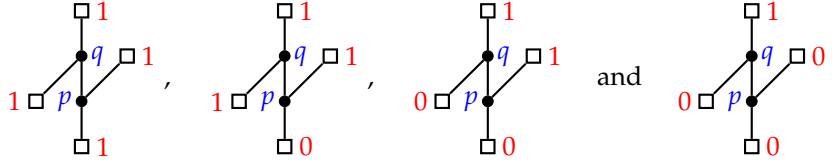
\begin{equation*}
        \begin{tikzpicture}[scale=.6,baseline={([yshift=-.5ex]current bounding box.center)}]]
            \path (0,0) node[posetelmm] (0) {} node[right=2pt,marking] {\(1\)};
            \path (0,1) node[posetelm] (x) {} node[left,elmname] {\(p\)};
            \path (0,2) node[posetelm] (y) {} node[right,elmname] {\(q\)};
            \path (0,3) node[posetelmm] (4) {} node[right=2pt,marking] {\(1\)};
            \path (-1,1) node[posetelmm] (1) {} node[left=2pt,marking] {\(1\)};
            \path (1,2) node[posetelmm] (3) {} node[right=2pt,marking] {\(1\)};
            \draw[covrel]
                  (0) -- (x) -- (y) -- (4)
                  (1) -- (y)
                  (x) -- (3);
        \end{tikzpicture},\quad
        \begin{tikzpicture}[scale=.6,baseline={([yshift=-.5ex]current bounding box.center)}]]
            \path (0,0) node[posetelmm] (0) {} node[right=2pt,marking] {\(0\)};
            \path (0,1) node[posetelm] (x) {} node[left,elmname] {\(p\)};
            \path (0,2) node[posetelm] (y) {} node[right,elmname] {\(q\)};
            \path (0,3) node[posetelmm] (4) {} node[right=2pt,marking] {\(1\)};
            \path (-1,1) node[posetelmm] (1) {} node[left=2pt,marking] {\(1\)};
            \path (1,2) node[posetelmm] (3) {} node[right=2pt,marking] {\(1\)};
            \draw[covrel]
                  (0) -- (x) -- (y) -- (4)
                  (1) -- (y)
                  (x) -- (3);
        \end{tikzpicture},\quad
        \begin{tikzpicture}[scale=.6,baseline={([yshift=-.5ex]current bounding box.center)}]]
            \path (0,0) node[posetelmm] (0) {} node[right=2pt,marking] {\(0\)};
            \path (0,1) node[posetelm] (x) {} node[left,elmname] {\(p\)};
            \path (0,2) node[posetelm] (y) {} node[right,elmname] {\(q\)};
            \path (0,3) node[posetelmm] (4) {} node[right=2pt,marking] {\(1\)};
            \path (-1,1) node[posetelmm] (1) {} node[left=2pt,marking] {\(0\)};
            \path (1,2) node[posetelmm] (3) {} node[right=2pt,marking] {\(1\)};
            \draw[covrel]
                  (0) -- (x) -- (y) -- (4)
                  (1) -- (y)
                  (x) -- (3);
        \end{tikzpicture}\quad\text{and}\quad
        \begin{tikzpicture}[scale=.6,baseline={([yshift=-.5ex]current bounding box.center)}]]
            \path (0,0) node[posetelmm] (0) {} node[right=2pt,marking] {\(0\)};
            \path (0,1) node[posetelm] (x) {} node[left,elmname] {\(p\)};
            \path (0,2) node[posetelm] (y) {} node[right,elmname] {\(q\)};
            \path (0,3) node[posetelmm] (4) {} node[right=2pt,marking] {\(1\)};
            \path (-1,1) node[posetelmm] (1) {} node[left=2pt,marking] {\(0\)};
            \path (1,2) node[posetelmm] (3) {} node[right=2pt,marking] {\(0\)};
            \draw[covrel]
                  (0) -- (x) -- (y) -- (4)
                  (1) -- (y)
                  (x) -- (3);
        \end{tikzpicture},
    \end{equation*}
    respectively.
    The associated marked order polytopes and their weighted Minkowski sum are
    \begin{equation*}
        \tikzset{el/.style={draw, fill, circle, minimum size=3pt, inner sep=0, outer sep=0}}
        0\,
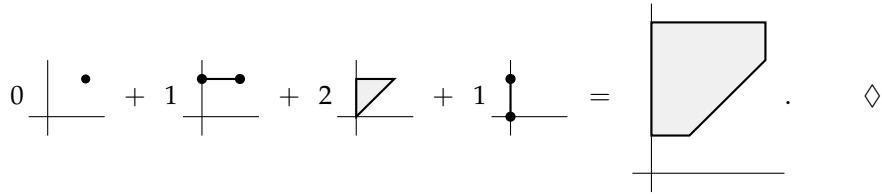
        \begin{tikzpicture}[scale=.5,baseline={([yshift=-.5ex]current bounding box.center)}]

            \draw (-.5,0) -- (1.5,0);
            \draw (0,-.5) -- (0,1.5);
            \node[el] at (1,1) {};
        \end{tikzpicture} \enskip+\enskip
        1\,
        \begin{tikzpicture}[scale=.5,baseline={([yshift=-.5ex]current bounding box.center)}]

            \draw (-.5,0) -- (1.5,0);
            \draw (0,-.5) -- (0,1.5);
            \draw (0,1) node[el] {} [thick] -- (1,1) node[el] {};
        \end{tikzpicture} \enskip+\enskip
        2\,
        \begin{tikzpicture}[scale=.5,baseline={([yshift=-.5ex]current bounding box.center)}]

            \draw (-.5,0) -- (1.5,0);
            \draw (0,-.5) -- (0,1.5);
            \fill[black!6,draw=black,thick] (0,0) -- (1,1) -- (0,1) -- cycle;
        \end{tikzpicture} \enskip+\enskip
        1\,
        \begin{tikzpicture}[scale=.5,baseline={([yshift=-.5ex]current bounding box.center)}]

            \draw (-.5,0) -- (1.5,0);
            \draw (0,-.5) -- (0,1.5);
            \draw (0,0) node[el] {} [thick] -- (0,1) node[el] {};
        \end{tikzpicture} \enskip=\enskip
        \begin{tikzpicture}[scale=.5,baseline={([yshift=-.5ex]current bounding box.center)}]

            \draw (-.5,0) -- (3.5,0);
            \draw (0,-.5) -- (0,4.5);
            \fill[black!6,draw=black,thick] (0,1) -- (1,1) -- (3,3) -- (3,4) -- (0,4) -- cycle;
        \end{tikzpicture}. \qedhere
    \end{equation*}
\end{example}

We finish this section by considering marked posets with integral markings.
When all markings on a poset \(P\) are integral, we will see that \(\OO(P,\lambda)\) is a lattice polyhedron.
A simple fact about lattice polyhedra we will need is that products of lattice polyhedra are lattice polyhedra.
This is an immediate consequence of the Minkowski sum identity \((Q+R)\times(Q'+R')=(Q\times Q')+(R\times R')\) we already used, together with products of lattice polytopes being lattice polytopes and products of rational cones being rational cones.

\begin{proposition} \label{prop:latticepoly}
    Let \((P,\lambda)\) be a marked poset such that \(\lambda(P^*)\subseteq \Z\).
    Then the marked order polyhedron \(\OO(P,\lambda)\) is a lattice polyhedron.
\end{proposition}

\begin{proof}
    When \(\OO(P,\lambda)\) is pointed, it is enough to show that all vertices are lattice points and the recession cone is rational.
    By \Cref{prop:mopface} the face partitions associated to vertices have no free blocks.
    Hence, all coordinates are contained in \(\lambda(P^*)\), so vertices are lattice points.
    The recession cone is obtained as \(\OO(P,0)\) by \Cref{prop:rec}, which is a rational polyhedral cone.

    If \(\OO(P,\lambda)\) is not pointed, we use the decomposition \(P=P'\sqcup P''\) of \(P\) into unmarked connected components in \(P'\) and the other components in \(P''\).
    As in the previous proof, we obtain a product decomposition \(\OO(P,\lambda)=\OO(P',\lambda')\times\OO(P'',\lambda'')\) by \Cref{prop:products}.
    Since \(P'^*\) is empty we know that \(\OO(P',\lambda')\) is a rational polyhedral cone.
    Since all connected components of \(P''\) contain marked elements, we know that \(\OO(P'',\lambda'')\) is pointed and hence a lattice polyhedron by the previous argument.
    We conclude that \(\OO(P,\lambda)\) is a lattice polyhedron.
\end{proof}

\section{Conditional Marked Order Polyhedra}
\label{sec:cmop}

In this section we study intersections of marked order polyhedra with affine subspaces.
We describe an affine subspace \(U\) of \(\R^P\) by a linear map \(s\colon \R^P\to \R^k\) and a vector \(b\in\R^k\), such that \(U=s^{-1}(b)\).
Hence, \(U\) is the space of solutions to the linear system \(s(x)=b\).

\begin{definition}
    Given a marked poset \((P,\lambda)\), a linear map \(s\colon \R^P\to\R^k\) and \(b\in\R^k\), we define the \emph{conditional marked order polyhedron} \(\OO(P,\lambda,s,b)\) as the intersection \(\OO(P,\lambda)\cap s^{-1}(b)\).
\end{definition}

The faces of \(\OO(P,\lambda,s,b)\) correspond to the faces of \(\OO(P,\lambda)\) whose relative interior meets \(s^{-1}(b)\).
Hence, they are also given by face partitions.
However, given a face partition \(\pi\) of \(\OO(P,\lambda)\), deciding whether it is a face partition of \(\OO(P,\lambda,s,b)\) can not be done combinatorially in general.
The problem is in determining whether the linear system \(s(x)=b\) admits a solution in the relative interior of \(F_\pi\).
We come back to this issue later in the section.
Still, given a point \(x\in\OO(P,\lambda,s,b)\), we obtain a face partition \(\pi_x\) and we can find the dimension of \(F_x\subseteq \OO(P,\lambda,s,b)\) by calculating a kernel of a linear map associated to \(\pi_x\).

Given a partition \(\pi\) of \(P\), we define the linear injection \(r_\pi\colon\R^{\tilde\pi}\to\R^P\) by
\begin{equation*}
    r_\pi(z)_p =
    \begin{cases}
        z_B & \text{if \(p\) is an element of the free block \(B\in\tilde\pi\),} \\
        0 & \text{otherwise.}
    \end{cases}
\end{equation*}
We can describe \(r_\pi\) as taking a labeling \(z\) of the free blocks of \(\pi\) with real numbers and making it into a labeling of \(P\) with real numbers, by putting the values given by \(z\) on elements in free blocks, while labeling elements in non-free blocks with zero.
If \(\pi\) is a face partition of \(\OO(P,\lambda)\), we have seen in the proof of \Cref{prop:mopface} that the affine hull of \(F_\pi\subseteq\OO(P,\lambda)\) is a translation of \(\im (r_\pi)\).
The following proposition is a generalization of this observation to conditional marked order polyhedra.

\begin{proposition} \label{prop:cmopface}
Let \(x\) be a point of \(\OO(P,\lambda,s,b)\) with associated face partition \(\pi=\pi_x\).
Let \(U\) be the linear subspace of \(\R^P\) parallel to the affine hull of the face \(F_x\subseteq\OO(P,\lambda,s,b)\).
The map \(r_\pi\) restricts to an isomorphism \(\ker (s\circ r_\pi)\isoto U\).
In particular, the dimension of \(F_x\) is the same as the dimension of \(\ker (s\circ r_\pi)\).
\end{proposition}

\begin{proof}
    Let \(F'_x\) be the minimal face of \(\OO(P,\lambda)\) containing \(x\), so that \(F_x=F'_x\cap s^{-1}(b)\).
    For the affine hulls we also have \(\aff(F_x) = \aff(F'_x) \cap s^{-1}(b)\).
    Letting \(U'\) be the linear subspace parallel to \(\aff(F'_x)\), just as \(U\) is the linear subspace parallel to \(\aff(F_x)\), we obtain
    \begin{equation*}
        U = U' \cap \ker(s) = \im(r_\pi) \cap \ker(s),
    \end{equation*}
    since \(\ker(s)\) is the linear subspace parallel to \(s^{-1}(b)\).
    This description implies that \(r_\pi\) restricts to an isomorphism \(\ker (s\circ r_\pi)\isoto U\).
\end{proof}

\begin{remark}
    In the special case of Gelfand--Tsetlin polytopes with linear conditions given by a weight \(\mu\), this result appeared in \cite{DM04} in terms of tiling matrices associated to points in the polytope.
    The tiling matrix is exactly the matrix associated to the linear map \(s\circ r_\pi\).
\end{remark}

\begin{example} \label{ex:cmop}
    Let \((P,\lambda)\) be the linear marked poset
    \begin{equation*}
        \tikz[baseline=-3.5pt]{\node[marking] {\(0\)};}
        \prec
        \textcolor{blue}{p}
        \prec
        \textcolor{blue}{q}
        \prec
        \textcolor{blue}{r}
        \prec
        \textcolor{blue}{s}
        \prec
        \tikz[baseline=-3.5pt]{\node[marking] {\(5\)};}
    \end{equation*}
    and impose the linear conditions \(x_p+x_r=4\) and \(x_q+x_s=6\) on \(\OO(P,\lambda)\).
    We describe these conditions by intersecting with \(s^{-1}(b)\) for the linear map \(s\colon \R^P\to \R^2\) given by \(s(x)=(x_p+x_r, x_q+x_s)\) and \(b=(4,6)\).
    Any point in \(\OO(P,\lambda,s,b)\) is determined by \(x_p\) and \(x_q\), so we can picture the polytope in \(\R^2\).
    Expressing the five inequalities in terms of \(x_p\), \(x_q\) using the linear conditions, we obtain
    \begin{equation*}
        0\le x_p, \enskip
        x_p \le x_q, \enskip
        x_q \le 4-x_p, \enskip
        x_q \le 2+x_p, \enskip
        1 \le x_q.
    \end{equation*}
    The resulting polytope in \(\R^{\{p,q\}}\cong\R^2\) is illustrated in \Cref{fig:Lam52}.
    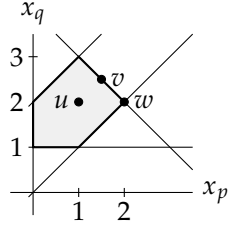
\begin{figure}
        \centering
        \begin{tikzpicture}[scale=.6]
            \tikzset{el/.style={draw, fill, circle, minimum size=3pt, inner sep=0, outer sep=0}}
            \draw (-.5,0) -- (3.5,0) node [right] {\(x_p\)};
            \draw (0,-.5) -- (0,3.5) node [above] {\(x_q\)};
            \draw (0.5,3.5) -- (3.5,0.5);
            \draw (-.1,1) -- (3.5,1);
            \draw (-.1,-.1) -- (3.5,3.5);
            \draw (-.1,1.9) -- (1.5,3.5);
            \fill[black!6,draw=black,thick] (0,1) -- (1,1) -- (2,2) -- (1,3) -- (0,2) -- cycle;
            \draw (1,-.1) -- (1,.1);
            \draw (1,0) node[below] {\(1\)};
            \draw (2,-.1) -- (2,.1);
            \draw (2,0) node[below] {\(2\)};
            \draw (0,1) node[left] {\(1\)};
            \draw (0,2) node[left] {\(2\)};
            \draw (-.1,3) -- (.1,3);
            \draw (0,3) node[left] {\(3\)};
            \draw (1,2) node[el] {} node[left] {\(u\)};
            \draw (1.5,2.5) node[el] {} node[right] {\(v\)};
            \draw (2,2) node[el] {} node[right] {\(w\)};
        \end{tikzpicture}
        \caption[Conditional marked order polytope with three points]{The conditional marked order polytope \(\OO(P,\lambda,s,b)\) from \Cref{ex:cmop} together with three points on faces of different dimensions.}
        \label{fig:Lam52}
    \end{figure}

    We want to calculate the dimensions of the minimal faces of \(\OO(P,\lambda,s,b)\) containing the points \(u=(1,2)\), \(v=(1.5, 2.5)\) and \(w=(2,2)\) in \(\R^2\).
    In \(\R^P\) these points and their associated partitions of \(P\) are
    \begin{equation*}
        \textcolor{red}{0} \mid
        \textcolor{blue}{1} \mid
        \textcolor{blue}{2} \mid
        \textcolor{blue}{3} \mid
        \textcolor{blue}{4} \mid
        \textcolor{red}{5}, \enskip
        \textcolor{red}{0} \mid
        \textcolor{blue}{1.5} \mid
        \textcolor{blue}{2.5} \enskip
        \textcolor{blue}{2.5} \mid
        \textcolor{blue}{3.5} \mid
        \textcolor{red}{5}, \enskip\text{and}\enskip
        \textcolor{red}{0} \mid
        \textcolor{blue}{2} \enskip
        \textcolor{blue}{2} \enskip
        \textcolor{blue}{2} \mid
        \textcolor{blue}{4} \mid
        \textcolor{red}{5}.
    \end{equation*}
    Hence, we have \(4\), \(3\) and \(2\) free blocks, respectively.
    The associated linear maps \(s\circ r_{\pi}\) can be represented by the matrices
    \begin{equation*}
        \begin{pmatrix}
            1 & 0 & 1 & 0 \\
            0 & 1 & 0 & 1
        \end{pmatrix}, 
        \begin{pmatrix}
            1 & 1 & 0 \\
            0 & 1 & 1
        \end{pmatrix}, \enskip\text{and}\enskip
        \begin{pmatrix}
            2 & 0 \\
            1 & 1
        \end{pmatrix}, 
    \end{equation*}
    respectively.
    The kernels of these maps have dimension \(2\), \(1\) and \(0\) corresponding to the dimensions of the minimal faces containing \(u\), \(v\) and \(w\) as one can see in \Cref{fig:Lam52}.
\end{example}

Given any \((P,\lambda)\)-compatible partition of \(P\), we obtained a polyhedron \(F'_\pi\) contained in \(\OO(P,\lambda)\) in the previous section.
Hence, we have a polyhedron \(F_\pi\) contained in \(\OO(P,\lambda,s,b)\) given by \(F_\pi = F'_\pi \cap s^{-1}(b)\).
As in the unconditional case, these polyhedra are canonically affine isomorphic to conditional marked order polyhedra given by the quotient \((P/\pi,\lambda/\pi)\).

\begin{proposition}
    Let \((P,\lambda)\) be a marked poset, \(\pi\) a \((P,\lambda)\)-compatible partition, \(s\colon\R^P\to\R^k\) a linear map and \(b\in\R^k\).
    Define \(s/\pi\) to be the composition \(s\circ q^*\), where \(q^*\) is the inclusion \(\R^{P/\pi}\hookrightarrow \R^P\) induced by the quotient map of marked posets.
    The polyhedron \(F_\pi\subseteq\OO(P,\lambda,s,b)\) is affinely isomorphic to the conditional marked order polyhedron \(\OO(P/\pi, \lambda/\pi, s/\pi, b)\) via the canonical isomorphism obtained by restricting \(q^*\).
\end{proposition}

\begin{proof}
    By definition, \(F_\pi\) is the intersection of the face \(F'_\pi\) of \(\OO(P,\lambda)\) with \(s^{-1}(b)\).
    We know that \(q^*\) restricts to an affine isomorphism \(\OO(P/\pi,\lambda/\pi)\isoto F'_\pi\).
    Hence, \(F_\pi\) is contained in the image of \(q^*\) as well and we have
    \begin{equation*}
        F_\pi = F'_\pi \cap s^{-1}(b) = F'_\pi \cap \im q^* \cap s^{-1}(b) = F'_\pi \cap q^*( (s\circ q^*)^{-1}(b)).
    \end{equation*}
    We may write \(F'_\pi\) as \(q^*(\OO(P/\pi,\lambda/\pi))\) and use injectivity of \(q^*\) to obtain
    \begin{equation*}
        F_\pi = q^*(\OO(P/\pi,\lambda/\pi))\cap q^*( (s/\pi)^{-1}(b) ) = q^*(\OO(P/\pi,\lambda/\pi)\cap (s/\pi)^{-1}(b)).
    \end{equation*}
    By definition of conditional marked order polyhedra, this is just the injective image of \(\OO(P/\pi,\lambda/\pi,s/\pi,b)\) under \(q^*\), which finishes the proof.
\end{proof}

When \(F\) is a non-empty face of \(\OO(P/\pi,\lambda/\pi,s/\pi,b)\) we have an associated partition \(\pi=\pi_F\), so that \(F=F_\pi\).
Thus, we obtain the same corollary on faces of conditional marked order polyhedra as in the unconditional case.

\begin{corollary}
    For every non-empty face \(F\) of a conditional marked order polyhedron \(\OO(P,\lambda,s,b)\) we have a canonical affine isomorphism
    \begin{equation*}
        \pushQED{\qed}
        \OO(P/\pi_F, \lambda/\pi_F, s/\pi_F, b) \cong F. \qedhere
        \popQED
    \end{equation*}
\end{corollary}

The next proposition will allow us to consider \emph{any} polyhedron as a conditional marked order polyhedron up to affine isomorphism.
Thus, there is little hope to understand general conditional marked order polyhedra any better than we understand polyhedra in general.

\begin{proposition} \label{prop:allcmops}
    Every polyhedron is affinely isomorphic to a conditional marked order polyhedron.
\end{proposition}

\begin{proof}
    Let \(Q\subseteq \R^n\) be a polyhedron given by linear equations and inequalities
    \begin{align*}
        \sum_{i=1}^n a_{ki} x_i &= c_k &\text{for \(k\)}&=1,\dots,s, \\
        \sum_{i=1}^n b_{li} x_i &\le d_l &\text{for \(l\)}&=1,\dots,t.
    \end{align*}
    Define a poset \(P\) with ground set \(\{p_1,\dots,p_n,q_1,\dots,q_t,r\}\) and covering relations \(q_l \prec r\) for \(l=1,\dots,t\).
    Define a marking on \(P^*=\{r\}\) by \(\lambda(r)=0\).
    The marked poset obtained this way is depicted in \Cref{fig:helpmop}.
    Let the linear system \(s(x)=b\) for \(x\in \R^P\) be given by
    \begin{align*}
        \sum_{i=1}^n a_{ki} x_{p_i} &= c_k &\text{for \(k\)}&=1,\dots,s, \\
        \sum_{i=1}^n b_{li} x_{p_i} - x_{q_l} &= d_l &\text{for \(l\)}&=1,\dots,t.
    \end{align*}
    The conditional marked order polyhedron \(\OO(P,\lambda,s,b)\) is affinely isomorphic to \(Q\) by the map \(\OO(P,\lambda,s,b)\to Q\) sending \(x\in\R^P\) to \((x_{p_1},\dots,x_{p_n})\in\R^n\).
    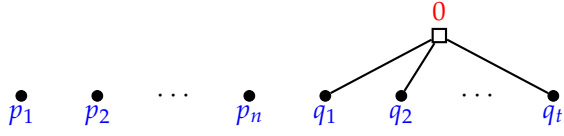
\begin{figure}
        \centering
        \begin{tikzpicture}[yscale=.8]
            \draw (0,0) node[posetelm] {} node[below, elmname] {\(p_1\)}
                  (1,0) node[posetelm] {} node[below, elmname] {\(p_2\)}
                  (2,0) node {\(\cdots\)}
                  (3,0) node[posetelm] {} node[below, elmname] {\(p_n\)}
                  (4,0) node[posetelm] (q1) {} node[below, elmname] {\(q_1\)}
                  (5,0) node[posetelm] (q2) {} node[below, elmname] {\(q_2\)}
                  (6,0) node {\(\cdots\)}
                  (7,0) node[posetelm] (qt) {} node[below, elmname] {\(q_t\)}
                  (5.5,1) node[posetelmm] (0) {} node[above=2pt,marking] {\(0\)};
              \draw[covrel] (q1) -- (0) (q2) -- (0) (qt) -- (0);
        \end{tikzpicture}
        \caption[Marked poset from proof of \Cref{prop:allcmops}]{The marked poset constructed in the proof of \Cref{prop:allcmops}.}
        \label{fig:helpmop}
    \end{figure}
\end{proof}

We may now come back to the question of when a face partition \(\pi\) of \((P,\lambda)\) still corresponds to a face of \(\OO(P,\lambda,s,b)\).
As discussed at the beginning of this section, we have to decide whether \(s(x)=b\) admits a solution in the relative interior of the face \(F'_\pi\) of \(\OO(P,\lambda)\), that is, \(\relint(F'_\pi) \cap s^{-1}(b) \neq \varnothing\).
Using the affine isomorphism induced by the quotient map this is equivalent to
\begin{equation*}
    \relint\left( \OO(P/\pi, \lambda/\pi) \right) \cap (s/\pi)^{-1}(b) \neq \varnothing.
\end{equation*}
Hence, we reduced the problem to deciding whether a linear system \(s(x)=b\) admits a solution in the relative interior of a marked order polyhedron \(\OO(P,\lambda)\).
However, even deciding whether \(s(x)=b\) admits any solution in \(\OO(P,\lambda)\) is equivalent to deciding whether \(\OO(P,\lambda,s,b)\) is non-empty, which is in general just as hard as determining whether an arbitrary system of linear equations and linear inequalities admits a solution by \Cref{prop:allcmops}.

We conclude that the concept of conditional marked order polyhedra is too general to obtain meaningful results.
Still, in special cases the structure of an underlying poset and faces still corresponding to a subset of face partitions might be useful.
An interesting class of conditional marked order polyhedra might consist of those, where \(P\) is connected and conditions are given by fixing sums along disjoint subsets of \(P\), as is the case for Gelfand--Tsetlin polytopes with weight conditions.

\printbibliography

\end{document}